\documentclass[12pt]{amsart}

\usepackage[active]{srcltx}

\usepackage{amssymb}
\usepackage{amsmath}
\usepackage{amsfonts}
\usepackage{bbm}
\usepackage[all]{xy} % For commutative diagrams
\usepackage{color}

\newcommand{\disk}{\ensuremath{\mathbb{D}} } % unit disk
\newcommand{\sphere}{\bar{\Bbb{C}}} %Riemann sphere
\newcommand{\riem}{\Sigma}  %Riemann surface
\renewcommand{\Bbb}[1]{\ensuremath{\mathbb{#1}}}
\newcommand{\st}{\, | \,} % such that
\newcommand{\Oqc}{\mathcal{O}^{\mathrm{qc}}} % Oqc
\newcommand{\qs}{\operatorname{QS}}

\newcommand{\Oqcwp}{\Oqc_{\operatorname{WP}}}

\newcommand{\qswp}{\operatorname{QS}_{\mathrm{WP}}}

\newcommand{\qcr}{\operatorname{QC}_r}
\newcommand{\qc}{\operatorname{QC}}
\newcommand{\qco}{\operatorname{QC}_0}

\newcommand{\bd}{\operatorname{BD}}
\newcommand{\tbd}{\operatorname{TBD}}

\newcommand{\twp}{T_{\mathrm{WP}}}
\newcommand{\ttwp}{\widetilde{T}_{\mathrm{WP}}}

\newcommand{\teich}{Teichm\"uller }

\theoremstyle{plain}
        \newtheorem{theorem}{Theorem}[section]
        \newtheorem{lemma}[theorem]{Lemma}
        \newtheorem{proposition}[theorem]{Proposition}
        \newtheorem{corollary}[theorem]{Corollary}
          
\theoremstyle{definition}
        \newtheorem{definition}[theorem]{Definition}

\theoremstyle{remark}
    \newtheorem{remark}[theorem]{Remark}

\numberwithin{equation}{section} % Equation labels are 'section'.'eq #'
\numberwithin{figure}{section} % Figures labela are 'section.'fig #'

\usepackage{fullpage}

\title{A convergent Weil-Petersson metric on the Teichm\"uller space of bordered Riemann surfaces}

\author{David Radnell}
\address{David Radnell \\ Department of Mathematics and Statistics \\
American University of Sharjah \\
PO BOX 26666, Sharjah \\ United Arab Emirates} \email{dradnell@aus.edu}

\author{Eric Schippers}
\address{Eric Schippers \\ Department of Mathematics \\
University of Manitoba\\
Winnipeg, Manitoba \\  R3T 2N2 \\ Canada}
\email{eric\_schippers@umanitoba.ca}

\author{Wolfgang Staubach}
\address{Wolfgang Staubach\\ Department of Mathematics\\
Uppsala University\\
Box 480\\ 751 06 Uppsala\\ Sweden}
\email{W.Staubach@hw.ac.uk}


\thanks{ Eric Schippers is partially supported by the National Sciences and Engineering Research
Council.}
\begin{document}
\maketitle
\begin{abstract} Let $\riem$ be a Riemann surface of genus $g$ bordered
by $n$ curves homeomorphic to the circle $\mathbb{S}^1$, and assume that
 $2g+2-n>0$. For such bordered Riemann surfaces, the authors have previously defined a Teichm\"uller space which is a Hilbert manifold and which is holomorphically included in the standard Teichm\"uller space. Based on this, we present alternate models of the aforementioned Teichm\"uller space and show in particular that it is locally modelled on a Hilbert space of
 $L^2$ Beltrami differentials, which are holomorphic up to a power of the hyperbolic metric, and has a convergent Weil-Petersson metric.
\end{abstract}

\begin{section}{Introduction}
\begin{subsection}{Introduction and statement of results}

In this paper, we demonstrate that the refined Teichm\"uller space of bordered Riemann surfaces defined
 by the authors in \cite{RSS_Hilbert} possesses a convergent Weil-Petersson metric, and a simple
 $L^2$ space of Beltrami differentials models the tangent space.

 The Weil-Petersson metric converges on finite-dimensional Teichm\"uller spaces.  However it has long
 been known to diverge on more general Teichm\"uller spaces. S. Nag and A. Verjovsky \cite{NagVerjovsky} showed
 that the metric does in fact converge in directions tangent to the subset of the universal Teichm\"uller
 space which correspond to analytic parametrizations of $\mathbb{S}^1$.  Later, G. Cui \cite{Cui} and G. Hui
 \cite{GuoHui}, and independently L. Takhtajan and L-P.~Teo \cite{Takhtajan_Teo_Memoirs}, found the completion of
 this space which is modelled on $L^2$ Beltrami differentials, or the space of univalent functions
 whose pre-Schwarzian is in the Bergman space of the unit disk. Takhtajan and Teo showed that the WP-class
 Teichm\"uller space is a Hilbert manifold locally modelled on a certain space of
 quadratic differentials and is a topological group.  They also
 gave explicit K\"ahler potentials for the WP-metric, among many other results.

 Since then there has been great interest in what has come to be called
 the WP-class universal Teichm\"uller space.  A partial list of
 references can be found in the introduction to the paper of Y. Shen \cite{Shen_characterization}.  The paper
 of Shen also contains a solution to the problem of intrinsically characterizing ``WP-class quasisymmetries'';
 that is, the restricted class of quasisymmetries whose resulting conformal welding maps have $L^2$ pre-Schwarzians.
 This gave another characterization of the WP-class universal Teichm\"uller space.

 One might think that these results on the universal Teichm\"uller space can be passed down to
 arbitrary quotients to obtain arbitrary Teichm\"uller spaces.  However, this is not the
 case.  Given a bordered Riemann surface, the relevant class of quadratic differentials are $L^2$ on the surface
 itself, and hence on the fundamental domain of the associated Fuchsian group.  The lift of such a quadratic
 differential to the universal cover is not in $L^2$ unless the $L^2$ integral
 is zero on each fundamental domain (see Remarks \ref{re:no_lift_one} and \ref{re:no_lift_two}
 ahead).

 In this paper, we show that the refined \teich space of surfaces of genus $g$ bordered
 by $n$ curves homeomorphic to $\mathbb{S}^1$ (for $n>0$) possesses a convergent WP-metric.  To do this,
 we use a new technique.  We take advantage of a fiber structure discovered by two of the authors \cite{RadnellSchippers_fiber},
 which arises from an identification of the Teichm\"uller space of bordered surfaces with
 a moduli space of Friedan-Shenker-Vafa \cite{RadnellSchippers_monster}.  The authors
 used this fiber structure in \cite{RSS_Hilbert} to construct a Hilbert manifold structure
 on the refined Teichm\"uller space of bordered surfaces of this type. In this paper we
 show that this Hilbert manifold is locally modelled on a space of $L^2$ Beltrami
 differentials on the base surface that are holomorphic up to a power of the hyperbolic metric. Furthermore we will express the Weil-Petersson inner product in terms of these differentials. Our arguments rely heavily on sewing techniques developed in \cite{RadnellSchippers_monster} and \cite{RadnellSchippers_fiber}. 

 The results of this paper confirm that the refined Teichm\"uller space defined by
 the authors in \cite{RSS_Hilbert} is a natural object.  They also justify
 renaming this refined space the ``Weil-Petersson class'' Teichm\"uller space, and we
 will adopt that terminology in this paper.
\end{subsection}
\begin{subsection}{Definition of some function spaces} \label{se:differentials}
 First we establish some notation for the various function spaces involved.
 It is convenient here to not have to refer directly to the lift, as is the usual practice.  The
 (obviously) equivalent definitions can be found for example in \cite{Nagbook}.

 Let $\riem$ be a Riemann surface with a hyperbolic metric.
 Let $\mathcal{U}$ be an open covering of the Riemann surface $\riem$ by open sets $U$, each of which possesses
 a local parameter $\phi_U: U \rightarrow \mathbb{C}$ compatible with the complex structure.
 For $k,l \in \mathbb{Z}$ a $(k,l)$-differential $h$ is a collection of functions $\{ h_U:\phi_U(U) \to \mathbb{C} \,:\, U \in \mathcal{U} \}$ such that
 whenever $U \cap V$ is non-empty, denoting by $z=g(w)=\phi_V \circ \phi_U^{-1}(w)$ the change of parameter,
 the functions $h_U$ and $h_V$ satisfy the transformation rule
 \begin{equation} \label{eq:ndiff_trans}
  h_V(w) g'(w)^k \overline{g'(w)}^l = h_U(z);
 \end{equation}
 that is, $h$ has the expression
 $h_U(z) dz^k d\bar{z}^l$ in local coordinates.  For example, a Beltrami differential, or $(-1,1)$ differential, is a collection of functions satisfying
the transformation rule
 \begin{equation} \label{eq:Beltdiff_trans}
   h_V(w)\frac{\overline{g'(w)}}{g'(w)} = h_U(z);
 \end{equation}
 that is, $h$ has the expression $h_U(z) d\bar{z}/dz$ in local coordinates.  Similarly a quadratic differential is a $(2,0)$ differential
 and a function is a $(0,0)$ differential.

 We will be concerned with those differentials which are $L^p$ with respect to the hyperbolic metric for some
 $p$ (in this paper, we always have either $p=1$, $p=2$, or $p=\infty$).  Denote the expression for the hyperbolic
 metric $g$ in local coordinates by $\rho_U(z)^2 \,|dz|^2$ for a strictly positive function $\rho_U$; thus the metric
 transforms according to the rule
 \begin{equation} \label{eq:metric_trans}
 \rho_V(w) |g'(w)|= \rho_U(z) .
 \end{equation}
 Thus, if $W$ is an open set, which we momentarily
 assume to be entirely contained in some $U \in \mathcal{U}$,  for a $(k,l)$-differential
 we can define an $L^p$ integral with respect to the hyperbolic metric by
 \[  \| h \|^p_{p,\riem,W} =  \int_{\phi_U(W)} |h_U(z)|^p \rho_U(z)^{2-mp}, \]
 where $m=k+l$ and the right-hand integral is taken with respect to Lebesgue measure in the plane.
 It is easily checked that if $W$ is entirely contained in another open set $V \in \mathcal{U}$, then
 the integral obtained using $\phi_V$, $h_V$ and $\rho_V$ as above is identical, by (\ref{eq:ndiff_trans}), (\ref{eq:metric_trans})
 and a change of variables.

 By the standard construction using a partition of unity subordinate to the open cover $\mathcal{U}$,
 one can define the norm $\|h\|_{p,\riem,W}$
 on any open set $W \subseteq \riem$, including $W=\riem$.
 Similarly, one can define an $L^\infty$ norm
 \[  \| h \|_{\infty,\riem,W} := \| |h_U(z)| \rho_U^{-m} \|_\infty  \]
 for open sets $W$ in a single chart where the right hand side is the standard essential supremum with respect to Lebesgue
 measure.  As above this extends to any open subset $W \subseteq \riem$.
 \begin{definition} \label{de:Lp_definition}
  Let $W \subset \riem$ be an open set. For $1\leq p\leq \infty$, let
  \[  L^p_{k,l}(\riem,W) = \{(k,l)-\textrm{differentials} \,\,\, h \, :\, \| h \|_{p,W} <\infty \}.  \]
  Let
  \[  A^p_k(\riem,W) = \{ h \in L^p_{k,0}(\riem,W) \,: \, h \ \ \mbox{holomorphic} \}.  \]
  Denote $L^p_{k,l}(\riem,\riem)$ by $L^p_{k,l}(\riem)$ and $A^p_k(\riem,\riem)$ by $A^p_k(\riem)$.
 \end{definition}

 It will always be understood that any $L^p$ norm arising here is with respect to the
 hyperbolic metric.  Indeed, one cannot define the norm in general without the use
 of some invariant metric, except in special cases (e.g. for $k=2$, $l=0$ and $p=1$).
 \begin{remark} \label{re:Ldefinition_W_omit}
  We will not distinguish the norms $\| \cdot \|_{p,W}$ notationally with respect to
  the order of the differential, since the type of differential uniquely determines the norm.
  If the subscript ``$W$'' is omitted, it is assumed that $W=\riem$.
 \end{remark}
\end{subsection}
\end{section}
\begin{section}{WP-class Teichm\"uller space}
\begin{subsection}{WP-class quasisymmetries and quasiconformal maps}
 In \cite{RSS_Hilbert} the authors defined a Teichm\"uller space of bordered surfaces which possesses
 a Hilbert manifold structure.  We briefly review some of the definitions and results, as well as introduce
 new definitions necessary in the next few sections.

 Let
 \[  \disk = \{ z\,:\, |z| <1 \},  \ \ \ \disk^* = \{ z\,: \,|z|>1 \} \cup \{ \infty\},\ \
  \mathrm{and} \ \ \sphere = \mathbb{C} \cup \{\infty \}.  \]
 \begin{definition}  Let $\Oqcwp$ denote the set of holomorphic one-to-one maps $f:\disk \rightarrow \sphere$
  such that $(f''(z)/f'(z))dz  \in A^2_1(\disk)$ and $f(0)=0$.
 \end{definition}
 By results of Takhtajan and Teo, the image of $\Oqcwp$ under the map
 \begin{equation} \label{eq:Oqco_norm_def}
  f \longmapsto \left(\frac{f''(z)}{f'(z)} , f'(0) \right)
 \end{equation}
 is an open subset of the Hilbert space $A^2_1(\disk) \oplus \mathbb{C}$ with the direct sum
 inner product \cite[Theorem 2.3]{RSS_Hilbert}.

 Elements of $\Oqcwp$ all have quasiconformal extensions to $\sphere$.  They arise as conformal maps
 associated to quasisymmetries in the following way.  Given a quasisymmetry $\phi:\mathbb{S}^1 \rightarrow \mathbb{S}^1$,
 by the Ahlfors-Beurling extension theorem, there exists a quasiconformal map $w:\disk^* \rightarrow \disk^*$
 such that $\left. w \right|_{\mathbb{S}^1} = \phi$.  This quasiconformal map has complex dilatation
 \[  \mu = \frac{\overline{\partial} f}{\partial f} \in L^\infty_{-1,1}(\disk^*).  \]
 Let $f^\mu$ be the solution to the Beltrami equation
 \[  \frac{\overline{\partial} f}{\partial f} = \hat{\mu}  \]
 where $\hat{\mu}$ is the Beltrami differential which equals $\mu$ on $\disk^*$ and $0$ on $\disk$.  We normalize
 $f^\mu$ so that $f^{\mu}(0)=0$, $f^\mu(\infty)=\infty$ and ${f^\mu}'(\infty)=1$ for definiteness.  We define
 \[  f_\phi = \left. f^\mu \right|_{\disk}.  \]
 It is a standard result in Teichm\"uller theory that $f_\phi$ is independent of the choice of quasiconformal
 extension $w$, and furthermore $f_\phi = f_\psi$ if and only if $\phi = \psi$ \cite{Lehtobook}, \cite{Nagbook}.

 \begin{definition}  Let $\qswp(\mathbb{S}^1,\mathbb{S}^1)$ denote the set of quasisymmetric mappings $\phi$ from $\mathbb{S}^1$
 to $\mathbb{S}^1$ such that $f_\phi \in \Oqcwp$.
 \end{definition}
 We have the following theorem of Hui \cite{GuoHui}.
 \begin{theorem} \label{th:GuoHui} $\phi \in \qswp(\mathbb{S}^1,\mathbb{S}^1)$ if and only if
  $\phi$ has a quasiconformal extension $w:\disk^* \rightarrow \disk^*$ with Beltrami differential $\mu \in L^2_{-1,1}(\disk^*)$.
 \end{theorem}
 Note that since $w$ is quasiconformal, its Beltrami differential automatically satisfies $\mu \in L^\infty_{-1,1}(\disk^*)$.
 We also have the following recent remarkable result of Shen
 \cite{Shen_characterization}, which answers a question posed by Takhtajan and Teo \cite[Remark 1.10]{Takhtajan_Teo_Memoirs}.
 \begin{theorem}
  A homeomorphism $\phi:\mathbb{S}^1 \rightarrow \mathbb{S}^1$ is in $\qswp(\mathbb{S}^1,\mathbb{S}^1)$ if and only
  if $\phi$ is absolutely continuous and $\log \phi' \in H^{3/2}(\mathbb{S}^1)$ where $H^{3/2}(\mathbb{S}^1)$ is the
  fractional $3/2$ Sobolev space.
 \end{theorem}
 Although this result is not needed in this paper, we mention it because it provides
 a direct intrinsic characterization for $\qswp(\mathbb{S}^1,\mathbb{S}^1)$.

 From now on, let $\riem$ be a Riemann surface of genus $g$ bordered by $n$ curves homeomorphic to $\mathbb{S}^1$;
 we will always assume that $n>0$ when the term ``bordered'' is used.  We clarify the meaning of ``bordered'' now.  
 It is assumed that the Riemann surface is bordered in the sense of Ahlfors and Sario \cite[II.1.3]{AhlforsSario}.
 That is, the closure $\overline{\riem}$ of $\riem$ is a Hausdorff topological space, together with a maximal atlas of charts from open
 subsets of $\overline{\riem}$ into relatively open subsets of the closed upper half plane in $\mathbb{C}$, such that
 the overlap maps are conformal on their interiors.  (In particular, these charts have a continuous extension to the boundary).
 Thus for each point $p$ on the boundary, there exists
 a chart from an open set $U$ onto a disc $D=\{ z\,: \, |z| <1 \text{ and } \operatorname{Im}(z) >0 \}$ and a conformal
 map $\phi$ of $U$ onto $D$, such that $\phi$ extends homeomorphically to a relatively open set $\hat{U} \subset \overline{\riem}$
 which takes a segment of the boundary containing $p$ in its interior to a line segment in the plane.
 We will refer to such a chart $(\phi,U)$ as an ``upper half plane boundary chart''.    In order to
 avoid needless proliferation of notation, we will not distinguish
 $\phi$ notationally from its continuous extension, nor $U$ from $\hat{U}$.

 We will further allow charts in the interior of $\riem$ which map onto open neighbourhoods of $0$ in $\mathbb{C}$.  We also allow boundary charts onto sets of the form $\{ z\,:\, |z|  \leq 1 \} \cap \{ z \,:\, |z - a | <r \}$ where $r<1$ and $|a|=1$
 with conformal overlap maps as with the half-plane charts.  We will refer to such a boundary chart as a ``disc boundary chart''. We refer to either a disc boundary chart or an upper half plane boundary chart as a ``boundary chart''.

 Finally, when we say that $\riem$ is bordered by $n$ curves homeomorphic to $\mathbb{S}^1$, we mean that the boundary
 $\partial \riem$ consists of $n$ connected components, each of which is homeomorphic to $\mathbb{S}^1$ when endowed
 with the relative topology inherited from $\overline{\riem}$.  To say that $\riem$ is of genus $g$ means that $\riem$
 is biholomorphic to a subset $\riem^B$ of a compact Riemann surface $\widetilde{\riem}$ of genus $g$ in such a way
 that the complement of $\overline{\riem^B}$ in $\widetilde{\riem}$ consists of $n$ disjoint open sets biholomorphic
 to $\disk$.  Equivalently, the double of $\riem$ has genus $2g + n$.

 With this terminology established we may now make the following definition.
 \begin{definition} We say $\riem$ is a bordered surface of type $(g,n)$ if it is a bordered surface of genus $g$
 bordered by $n$ boundary curves homeomorphic to $\mathbb{S}^1$, in the sense of the last three paragraphs.
 \end{definition}

 We will also need one more kind of chart at the boundary.  Let
 \[ \mathbb{A}_r=\{ z\,:\, 1<|z|<r \}.  \]  The following proposition is elementary.
 \begin{proposition} \label{pr:collar_chart_existence}
  Let $\riem$ be a bordered Riemann surface of genus $g$ bordered by $n$ curves $\partial_i \riem$, $i=1,\ldots,n$,
  homeomorphic to $\mathbb{S}^1$.
  For each $i$, there exists an open set $A \subset \riem$ and an annulus $\mathbb{A}_r$ such that
  \begin{enumerate}
   \item $\partial_i \riem$ is contained in the closure
  of $A$
   \item $\partial A \cap \left( \partial_i \riem \right)^c$ is compactly contained in $\riem$
    \item there is a conformal map $\zeta:A \rightarrow \mathbb{A}_r$ for some $r$.
  \end{enumerate}
  For any such $A$, $\mathbb{A}_r$, and $\zeta$, $\zeta$ has a homeomorphic extension to $A \cup \partial_i \riem$.

  Furthermore, $A$, $r$ and $\zeta$ can be chosen so that
  $\partial A \backslash \partial_i \riem$ is an analytic curve.  In that case $\zeta$
  has a homeomorphic extension to the closure of $A$, which takes $\overline{A}$
  onto the closed annulus $\overline{\mathbb{A}}_r$.
 \end{proposition}
 We call such a chart a ``collar chart'' of $\partial_i \riem$, and $A$ a ``collar'' of $\partial_i \riem$.

 We may now define WP-class quasisymmetries between boundary curves of bordered Riemann surfaces, as in \cite{RSS_Hilbert}.
 \begin{definition} \label{de:refined_qs_surfaces}
  Let $\riem_1$ and $\riem_2$ be bordered Riemann surfaces of type $(g_i,n_i)$ respectively, and let $C_1$ and $C_2$ be
  boundary curves of $\riem_1$
  and $\riem_2$ respectively.  Let $\qswp(C_1,C_2)$ denote the set of orientation-preserving
  homeomorphisms $\phi:C_1 \rightarrow C_2$ such that there are collared charts $H_i$ of $C_i$, $i=1,2$ respectively,
  and such that $\left. H_2 \circ \phi \circ H_1^{-1} \right|_{S^1} \in \qswp(\mathbb{S}^1, \mathbb{S}^1)$.
 \end{definition}
 \begin{remark}  The notation $\qswp(\mathbb{S}^1,C_1)$ will always be understood to refer to $\mathbb{S}^1$ as the boundary of an annulus
  $\mathbb{A}_r$ for $r>1$.  We will also write $\qswp(\mathbb{S}^1)=\qswp(\mathbb{S}^1,\mathbb{S}^1)$.
 \end{remark}

\begin{remark}
\label{re:quasisymmetry_def}
In \cite[Section 2.4]{RadnellSchippers_monster}, Definition \ref{de:refined_qs_surfaces} was given, but with $\qswp(\mathbb{S}^1)$ replaced by standard quasisymmetries  $\qs(\mathbb{S}^1)$. These newly defined maps are also called quasisymmetries, and they have the usual relationship with quasiconformal maps.
\end{remark}

 We also have the following two properties of $\qswp(C_1,C_2)$ \cite{RSS_Hilbert}; the second is a
 consequence of a result of Takthajan and Teo \cite[Corollary 1.8]{Takhtajan_Teo_Memoirs}.
 \begin{proposition} \label{pr:rqs_chart_independent}
  If $\phi \in \qswp(C_1,C_2)$ then for any pair of collar charts $H_i$ of $C_i$, $i=1,2$ respectively,
  $\left. H_2 \circ \phi \circ H_1^{-1} \right|_{\mathbb{S}^1} \in \qswp(\mathbb{S}^1)$.
 \end{proposition}
 \begin{proposition} \label{pr:composition_preserves_qso_surfaces}
  Let $\riem^B_i$ be bordered Riemann surfaces and $C_i$ a boundary curve on each surface for $i=1,2,3$.  If $\phi \in \qswp(C_1,C_2)$
  and $\psi \in \qswp(C_2,C_3)$ then $\psi \circ \phi \in \qswp(C_1,C_3)$.
 \end{proposition}

  We will be concerned only with quasiconformal maps whose boundary values are in $\qswp$. Any such quasiconformal mapping has a homeomorphic
  extension taking the closure of $\riem_1$ to the closure of $\riem_2$.  This extension must map each boundary curve $\partial_i \riem_1$
  homeomorphically onto a boundary curve $\partial_j \riem_2$.
 \begin{definition} \label{de:qco}
  Let $\riem_1$ and $\riem_2$ be bordered Riemann surfaces of type $(g,n)$, with boundary curves $C_1^i$ and $C_2^j$,
  $i=1,\ldots,n$ and $j=1,\ldots,n$ respectively.  The class of maps $\qco(\riem_1,\riem_2)$ consists of
  those quasiconformal maps from $\riem_1$ onto $\riem_2$ such that the continuous extension to each boundary curve $C_1^i$, $i=1,\ldots,n$
  is in $\qswp(C_1^i,C_2^j)$ for some $j \in \{1,\ldots,n\}$.
 \end{definition}
 \begin{remark}
  The continuous extensions to the boundary will be made without further comment.
  We will not make any notational distinction between a quasiconformal map $f$ and its continuous extension.
 \end{remark}

 The following two Propositions follow immediately from Definition \ref{de:qco} and Proposition \ref{pr:composition_preserves_qso_surfaces}.
 \begin{proposition} \label{pr:composition_preserves_qco_surfaces}
  Let $\riem_i$ $i=1,2,3$ be bordered Riemann surfaces of type $(g,n)$.  If $f \in \qco(\riem_1,\riem_2)$ and $g \in \qco(\riem_2,\riem_3)$
  then $g \circ f \in \qco(\riem_1,\riem_3)$.
 \end{proposition}
 \begin{proposition} \label{pr:mixed_composition_preserves}
  Let $\riem_1$ and $\riem_2$ be bordered Riemann surfaces.  Let $C_1$ be a boundary curve of $\riem_1$,
  $\phi \in \qswp(\mathbb{S}^1,C_1)$, $f \in \qco(\riem_1,\riem_2)$ and $C_2=f(C_1)$ be the boundary curve of $\riem_2$ onto
  which $f$ maps $C_1$.
  Then $f \circ \phi \in \qswp(\mathbb{S}^1,C_2)$.
 \end{proposition}

 For a quasiconformal map $f$ let $\mu(f)$ denote its Beltrami differential as above.  Theorem
 \ref{th:GuoHui} above motivates the following definition.
 \begin{definition} \label{de:qc_refined}
  Let $\mbox{QC}_r(\riem,\riem_1)$ be the set of $f \in \qco(\riem,\riem_1)$ such that
  for any $i$ and any collar chart $\zeta_i:A_i \rightarrow \mathbb{A}_{r_i}$ on a collar $A_i$
  of $\partial_i \riem$,
  \begin{equation} \label{eq:collar_estimate_qcr_def}
   \iint_{\mathbb{A}_{r_i}} \frac{\left|\mu(f \circ \zeta_i^{-1})\right|^2}{(1-|z|^2)^2} \,dA <\infty.
  \end{equation}
 \end{definition}
 This condition can be thought of as requiring that $\mu(f)$ be ``hyperbolically $L^2$ near $\partial_i \riem$''.
 The condition appears to depend on the choice of chart, and it is not immediately obvious whether this relates to
 whether or not $\mu(f) \in L^2_{-1,1}(\riem)$.
 We will show that the conditions are the same.  We will prove this in the next section, with the help of a
 general local characterization of hyperbolic $L^p$ spaces.
\end{subsection}
\begin{subsection}{Local characterization of hyperbolically $L^p$ differentials} \label{se:local_charact}
 In this section, we will show that the condition that a differential be hyperbolically $L^p$ can be
 expressed locally in terms of the hyperbolic metric of a half-chart.  In particular, we will
 prove the following theorems.  Denote the expression for the hyperbolic metric on the
 upper half plane by
 \[  \lambda_{\mathbb{H}}(z)^2 |dz|^2  \]
 where $\lambda_\mathbb{H}(z)= 1/\mbox{Im}(z)$.  Similarly on the disc the hyperbolic metric is
 \[  \lambda_\disk(z)^2 |dz|^2 \]
 where $\lambda_\disk(z)= 1/(1-|z|^2)$.
 \begin{theorem} \label{th:ndiff_local_charact} Let $\riem$ be a bordered Riemann surface of type
  $(g,n)$ and let $\alpha$ be a $(k,l)$-differential on $\riem$.  Fix $p \in [1,\infty]$.
  The following are equivalent.
  \begin{enumerate}
   \item $\alpha \in L^p_{k,l}(\riem)$.
   \item For each point $q \in \overline{\riem}$, there is a chart $(\phi,U)$ of a neighbourhood of
   $q$ into the upper half plane $\mathbb{H}$, such that if $\alpha$
    is $h_U(z) dz^k d\bar{z}^l$ in local coordinates and $m=k+l$, the estimate
    \begin{eqnarray} \label{eq:chart_condition_UHP}
     \iint_{\phi(U)} \lambda_{\mathbb{H}}^{2-mp}(z)|h_U(z)|^p <\infty, & & p \in [1,\infty) \\
     \| \lambda_{\mathbb{H}}(z)^{-m} h_U(z) \|_{\infty, \phi(U)} <\infty, & & p=\infty  \nonumber
    \end{eqnarray}
    holds for the particular choice of $p$.
    \item For each point $q \in \overline{\riem}$, there is a chart $(\phi,U)$ of a neighbourhood of
   $q$ into the unit disc $\disk$, such that if $\alpha$
    is $h_U(z) dz^k d\bar{z}^l$ in local coordinates and $m=k+l$, the estimate
    \begin{eqnarray} \label{eq:chart_condition_disk}
     \iint_{\phi(U)} \lambda_{\mathbb{D}}^{2-mp}(z)|h_U(z)|^p <\infty, & & p \in [1,\infty) \\
     \| \lambda_{\mathbb{D}}(z)^{-m} h_U(z) \|_{\infty, \phi(U)} <\infty, & & p=\infty  \nonumber
    \end{eqnarray}
    holds for the particular choice of $p$.
   \item For each boundary curve $\partial_i \riem$, there is a collar chart $(\phi,U)$ of $\partial_i \riem$ for
   which the estimate (\ref{eq:chart_condition_disk}) holds.
   \item For any collar chart $(\phi_i,U_i)$ of any boundary curve $\partial_i \riem$ the estimate
   (\ref{eq:chart_condition_disk}) holds.
  \end{enumerate}
 \end{theorem}
 \begin{remark}
   As the reader may have noticed, there is no need to weight with the factor $(1-|z|^2)$ in the interior,
  analytically.  This is because $(1-|z|^2)$ and $\rho_U$ are continuous and the weight has no effect on the integral or on the $L^\infty$ norm. Thus Theorem \ref{th:ndiff_local_charact} is in effect about boundary values. However, for purely stylistic reasons, we shall keep the formulation of the theorem as above.
 \end{remark}

 This theorem follows from an elementary estimate which we state as two lemmas.
  \begin{lemma}  \label{le:singularity_of_hyp_metric}
  Let $\riem$ be a bordered Riemann surface of type $(g,n)$.
  Let $q \in \overline{\riem}$.  There is a chart $(\zeta,U)$ in a neighbourhood of $q$, with the following
  property.  There is a disc $D \subset \zeta(U)$ centred on $\zeta(p)$ if $q \in \riem$
  or a relatively open half-disc in $\overline{\mathbb{H}}$ centred on $\zeta(q)$ if $q \in \partial \riem$, and a $K>0$ such that
  \begin{equation} \label{eq:metric_comparison}
   \frac{1}{K} \leq \left| \frac{\rho_U(z) }{\lambda_{\mathbb{H}}(z)} \right| \leq K
  \end{equation}
  for all $z \in D$.  Here $\rho_U(z)|dz|^2$ is the expression for the hyperbolic metric on $\riem$ in the local parameter. The same claim holds for the hyperbolic metric $\lambda_{\disk}$ on $\disk$ and disk charts.
 \end{lemma}
 \begin{proof}
  If $q \in \riem$, then choosing $U$ to be an open neighbourhood of $q$ with compact closure in $\riem$,
  the estimate follows
  immediately from the fact that $\rho_U$ and $\lambda_{\mathbb{H}}$ are continuous and non-vanishing.

  Fix $q \in \partial \riem$.
  First we show that there is at least one chart in which the claim holds.  Let $\pi:\mathbb{H}
  \rightarrow \riem$ be the covering of $\riem$ by the upper half plane.  There is a relatively open
  set $\hat{U}$ in $\overline{\riem}$ containing $q$ such that there is a single-valued
  branch $\phi=\pi^{-1}$ on $U= \hat{U} \cap \riem$.  $\phi=\pi^{-1}$ is an isometry so that we have $\rho_U(z)=\lambda_{\mathbb{H}}(z)$
  and thus the claim holds for this chart.

  Let $(\zeta,V)$ be any other chart in a neighbourhood of $q$; we may assume without loss of generality
  that $(\zeta,V)$ is an upper half plane boundary chart centred on $q$.  Let $H=\phi \circ \zeta^{-1}$
  on $\zeta(U \cap V)$.
  In that case $H$ maps an open interval on $\mathbb{R}$ containing $\zeta(q)$ to an open
   interval of $\mathbb{R}$ containing $\phi(q)$, so by Schwarz reflection $H$ has an analytic continuation to an
   open disc containing an
  open interval on $\mathbb{R}$ with $\zeta(q)$ in its interior. Similarly the same claim holds for $H^{-1}$.
  We have
  \[   \frac{\rho_V(z)}{\lambda_{\mathbb{H}}(z)} = \frac{\rho_U(H(z))|H'(z)|}{\lambda_{\mathbb{H}}(z)} =
       \frac{\rho_U(H(z))}{\lambda_{H}(H(z))} \frac{\lambda_{\mathbb{H}}(H(z)) |H'(z)|}{\lambda_{\mathbb{H}}(z)}   \]
  so it suffices to estimate ${\lambda_{\mathbb{H}} \circ H |H'|^2}/{\lambda_{\mathbb{H}}}$.

  Let $w=H(z)=u(z)+iv(z)$ for real functions $u$ and $v$, and let $z=x+iy$.
  We have that the hyperbolic metric is
  $\lambda_{\mathbb{H}}(z)= 1/y$.  Since $H$ is a biholomorphism, $H' \neq 0$.  We claim
  that $v_y \neq 0$ at
  $\zeta(q)$.  If not, we would have $u_x=v_y=0$ at $\zeta(q)$.  Furthermore since $H$
  maps an interval on $\mathbb{R}$ containing $\zeta(q)$ to an interval in $\mathbb{R}$,
  $v=0$ on this interval so $u_y=-v_x=0$ on an interval containing $\zeta(q)$.  Thus
  the Jacobian of $H$ at $\zeta(q)$ is zero, a contradiction.  We conclude that
  there is a neighbourhood of $\zeta(q)$ on which $v_y \neq 0$.  Using a Taylor
  series approximation in two variables, and the fact that $v(x,0)=v_x(x,0)=0$, we have that
  \begin{equation} \label{eq:easy_linear}
    C |y| \leq |v(x,y)| \leq D|y|
  \end{equation}
  for some constants $C,D>0$ on some open disc centred on $\zeta(p)$ whose closure
  is contained in the domain of $H$.  Furthermore since $H$ is a biholomorphism there are constants
  $0<E,F$ such that $E \leq H' \leq F$ on a possibly smaller open disc whose closure
  is contained in the domain of $H$.  Since $\lambda_{\mathbb{H}} \circ H(z)=1/v(z)$, by (\ref{eq:easy_linear})
  there
  is a $K>0$ such that
  \[  \frac{1}{K} \leq \frac{\lambda_{\mathbb{H}} \circ H |H'|}{\lambda_{\mathbb{H}}}  \leq K  \]
   on this disk.  This proves the claim.

   The estimate for $\lambda_\mathbb{D}$ can be easily obtained by applying a M\"obius transformation.
 \end{proof}
  Lemma \ref{le:singularity_of_hyp_metric} can be improved slightly to the following.
 \begin{lemma}  \label{le:collar_hyperbolic_singularity}
  Let $\riem$ be a bordered Riemann surface of type $(g,n)$ and let $(\zeta_i,U_i)$ be a collar
  chart of $\partial_i \riem$.  There is an annulus $\mathbb{A}_{r,1} \subseteq \zeta_i(U_i)$ with $\mathbb{A}_{r,1}:=\{z;\, r<|z|<1\}$ such that
  \begin{equation*}
   \frac{1}{K} \leq \left| \frac{\rho_{U_i}(z) }{\lambda_{\disk}(z)} \right| \leq K
  \end{equation*}
  for all $z \in \mathbb{A}_{r,1}$.
 \end{lemma}
 \begin{proof}
  Repeating the proof of Lemma \ref{le:collar_hyperbolic_singularity}, for every point
  $q \in \partial_i \riem$ one obtains an open half-disc $\{z\,:\, |z-\zeta_i(q)|<s_q \} \cap \disk$
  on which the estimate holds.  Since $\partial_i \riem$ is compact the claim follows.
 \end{proof}

 \begin{proof} (of Theorem \ref{th:ndiff_local_charact})
  To see that (2) implies (1), observe that by Lemma \ref{le:singularity_of_hyp_metric}
  and the fact that $\overline{\riem}$ is compact, there is a finite collection of
  charts $(\zeta_i,U_i)$ and discs or half-discs $D_i$ in $\mathbb{H}$ such that
  $\zeta_i^{-1}(D_i)$ cover $\riem$ and on which the estimate (\ref{eq:metric_comparison})
  holds.  Thus there are constants $C_i(m,p)$ such that $\rho_{U_i}(z)^{2-mp} \leq C_i(m,p) \lambda_{\mathbb{H}}(z)^{2-mp}$
  for $p \in [1,\infty)$ and $\rho_{U_i}(z)^{-m} \leq C_i(m,\infty) \lambda_{\mathbb{H}}(z)^{-m}$.  Thus
  for $p \in [1,\infty)$
  \[  \iint_{D_i} \rho_{U_i}(z)^{2-mp} |h_{U_i}(z)|^p \leq \iint_{D_i} C_i(m,p) \lambda_{\mathbb{H}}(z)^{2-mp}
  |h_{U_i}(z)|^p <\infty \]
  for all $i$ and for $p=\infty$
  \[  \| \rho_{U_i}(z)^{-m} h_{U_i}(z) \|_{\infty,D_i} < C_i(m,\infty) \| \lambda_{\mathbb{H}}(z)^{-m} h_{U_i}(z) \|_{\infty,D_i}
   <\infty. \]
   Choosing a partition of unity subordinate to the finite covering proves (1).

   Now we show that (2) follows from (1).  For any point $q$ let $(\zeta,V)$ be any chart in a neighbourhood of $q$,
   and let $D$ be as in Lemma \ref{le:singularity_of_hyp_metric}.  We then have that there are constants
   $C(m,p)$ such that on $D$, $\lambda_{\mathbb{H}}(z)^{2-mp} \leq C(m,p) \rho_U(z)^{2-mp}$ for $p \in [1,\infty)$
   and $\lambda_{\mathbb{H}}(z)^{-m} \leq C(m,\infty) \rho_U(z)^{-m}$ for $p=\infty$.
   Set $U=\zeta^{-1}(D)$ now and let $\phi$ be the restriction of $\zeta$ to $U$; we then have
   \[     \iint_{\phi(U)} \lambda_{\mathbb{H}}^{2-mp}(z)|h_U(z)|^p < C(m,p) \iint_{\phi(U)}
   \rho_U^{2-mp}(z)|h_U(z)|^p \leq \| \alpha \|_{p,\riem} <\infty \]
   in the case that $p \neq \infty$ and
   \[   \| \lambda_{\mathbb{H}}(z)^{-m} h_U(z) \|_{\infty, \phi(U)} <  \| \rho_U(z)^{-m} h_U(z) \|_{\infty, \phi(U)}
          \leq \| \alpha \|_{\infty, \riem} < \infty \]
   in the case that $p=\infty$.

   The equivalence of (3) and (1) follows from an identical argument.  Clearly (5) implies (4) and
   (4) implies (3); on the other hand, if (1) holds, an argument similar to the proof of (2) above
   using Lemma \ref{le:collar_hyperbolic_singularity} establishes (5) (note that by definition
   the inner boundary of a collar chart is compactly contained in $\riem$).
 \end{proof}

 Finally, we will need the following lemma explicitly separating out the contribution
 of the collar to the $L^2$ norm.  We will only need the $p=2$ case, but since the
 general case requires no extra work, we will state it in general.
 \begin{lemma} \label{le:Lp_separation}
  Let $\riem$ be a bordered Riemann surface of genus $g$ with $n$ boundary curves.
  Fix $p \in [1,\infty)$.  Let $(\zeta,U)$ be a collection of collar charts
  $(\zeta_i,U_i)$ into $\disk$ for each boundary $i=1,\ldots,n$.  There exist
  annuli $\mathbb{A}_{r_i,1} = \{ z \,:\,r_i<|z|<1 \} \subset \zeta_i(U_i)$
  such that $|z|=r_i$ is compactly contained in $\zeta_i(U_i)$, a compact
  set $M$ such that
  \[  M \cup \zeta_1^{-1}(\mathbb{A}_{r_1,1}) \cup \cdots \cup \zeta_n^{-1}(\mathbb{A}_{r_n,1}) = \riem,   \]
  and constants $a$ and $b_i$ such that for any $\alpha \in L^p_{k,l}(\riem)$
  \[  \| \alpha \|_p \leq a \| \alpha\|_{\infty,M} + \sum_{i=1}^n b_i\left(  \iint_{\mathbb{A}_{r_i,1}}
  \lambda_{\mathbb{D}}^{2-mp}(z)|\alpha_{U_i}(z)|^p \right)^{1/p}.  \]
  The constants $b_i$ depend only on the collar charts
  $(\zeta,U)$, $r_i$, $p$, $k$ and $l$ $($not on $\alpha$$)$, and $a^p$ is the
  hyperbolic area of $M$.
 \end{lemma}
 \begin{proof}  Once the annuli are chosen, one need only choose $M$ such that
  $M$ together with $\zeta_i^{-1}(\mathbb{A}_{r_i,1})$ cover $\riem$.  The estimates on
  $\zeta_i^{-1}(\mathbb{A}_{r_i,1})$ follow from Lemma \ref{le:collar_hyperbolic_singularity},
  as in the proof of Theorem \ref{th:ndiff_local_charact}.

  The estimate on $M$ is obtained as follows.  Let $(\xi_j,W_j)$, $j=1,\ldots,N$ be charts into $\disk$, such that
  the open sets $W_j$ form an open cover of $M$.  Let $\chi_j$ be a partition of unity of $M$ subordinate
  to this covering (that is, $\sum \chi_j=1$ on $M$ and $\chi_j$ are supported in $W_j$).  Let $\mathbbm{1}_M$
  denote the characteristic function of $M$.  Then
  \begin{align*}
   \| \alpha \|^p_{p,M} & = \sum_{j=1}^N  \iint_{W_j} \chi_j\mathbbm{1}_M | \alpha_{W_j}(z)|^p\rho_{W_j}(z)^{2-mp} \\
   & \leq \sum_{j=1}^N  \| \alpha\|_{\infty,W_j}^p \; \iint_{W_j} \chi_j \mathbbm{1}_M \cdot \rho_{W_j}(z)^2 \\
   & \leq a^p \| \alpha \|^p_{\infty,M}.
  \end{align*}
  Since $p \geq 1$ the claim follows from the elementary inequality $\left( \sum a_k \right)^{1/p} \leq \sum a_k^{1/p}$.
 \end{proof}
\end{subsection}
\begin{subsection}{Quasiconformal maps with Beltrami differentials in $L^2_{-1,1}(\riem)$}
\label{se:L2Hyp_qcmaps}
 With the aid of the local characterization of hyperbolic $L^p$ spaces in Section \ref{se:local_charact},
 we can now generalize Theorem \ref{th:GuoHui} of Hui.  First, we make the following definition.
 \begin{definition}  Let
   \[  \tbd(\riem) = L^\infty_{-1,1}(\riem) \cap L^2_{-1,1}(\riem)  \]
   and
   \[  \bd(\riem) = \{ \mu \in \tbd(\riem) \,:\, \| \mu\|_{\infty,\riem}  \leq K \ \ \mbox{for some} \ \ K<1\}.  \]
 \end{definition}
 The ``BD'' in the above notation stands for ``Beltrami differentials''.
 The ``T'' stands for ``tangent''.  Analytically the notation ``$\tbd$'' is slightly inaccurate, since
 $\bd(\riem)$ is not a Hilbert or Banach linear space or manifold, so that it does not have
 a tangent space in the standard sense.  Nevertheless the notation distinguishes the spaces
 conveniently in terms of their upcoming roles. We thus have that
 \begin{theorem}  \label{th:QSR_is_L2Hyp} Let $\riem$ and $\riem_1$ be bordered Riemann surfaces
  of type $(g,n)$.
  Let $f:\riem \rightarrow \riem_1$ be quasiconformal with Beltrami differential $\mu(f)$.
  Then $f \in \mbox{QC}_r(\riem,\riem_1)$ if and only if $\mu(f) \in \bd(\riem)$.
 \end{theorem}
 \begin{proof}
  It is clear that if $\mu(f) \in \bd(\riem)$ then $f \in \qc_r(\riem,\riem_1)$. Now assume that $f \in \qcr(\riem,\riem_1)$.  By Definition \ref{de:qc_refined}, for each boundary curve $\partial_i \riem$
  there exists a collar chart $(\zeta_i,U_i)$, mapping onto $\mathbb{A}_{r_i}$ say, for
  which the estimate (\ref{eq:collar_estimate_qcr_def}) holds.  Clearly we can choose $U_i$ so that they
  are  pairwise disjoint.  Choose $1<R_i<r_i$.  Let
  \[  V=\riem \backslash \cup_{i=1,\ldots,n}  \{  \zeta_i^{-1}( \overline{ \mathbb{A}_{R_i}})  \}.  \]
  Observe that $V$ together with the $U_i$ cover $\riem$.

  Now $V$ is compactly contained in $\riem$.  Thus, for each point $q \in V$ there is a
  chart $(\eta,W)$ in a neighbourhood of $q$ defined on a pre-compact subset $W$ of
  $\riem$, which furthermore maps into a pre-compact subset of $\disk$.
  Since $\lambda_{\mathbb{D}}$ is continuous on the closure
  of $\eta(W)$ and the closure of $\eta(W)$ is compact, it is
  bounded there.  Since, $\|\mu(f)|_{W} \| \leq K$, we have that for $k=-1$, $l=1$
  condition (3) of
  Theorem \ref{th:ndiff_local_charact} is satisfied at $q$.
  Since the estimate (\ref{eq:collar_estimate_qcr_def}) holds on $\zeta_i(U_i)$
  and $V,U_1,\ldots,U_n$ cover $\riem$, by Theorem \ref{th:ndiff_local_charact}
  $\mu(f) \in \bd(\riem)$.
 \end{proof}

 We prove two more theorems relating to the existence of elements of $\qcr$ in a
 fixed homotopy class. They ultimately rely on the extended lambda-lemma \cite{Slodkowski}, through
 \cite[Lemma 4.2]{RadnellSchippers_monster}.
 
 The following theorem is not strictly necessary for the main results of this paper, however since it is similar to Theorem \ref{th:qcr_in_Teich_equivalence_class} ahead and will also be useful for the study of quasiconformal mapping class groups in $\twp(\riem),$ we include it here.
 
 \begin{theorem} \label{th:qcr_map_existence}  Let $\riem$, $\riem_1$ be bordered Riemann surfaces of type
 $(g,n)$ and $f:\riem \rightarrow \riem_1$ be a quasiconformal map.  Let $\phi_i \in \qswp(\partial_i \riem,\partial_i \riem_1)$
 for $i=1,\ldots,n$ where $\partial_i \riem_1 = f(\partial_i \riem)$.  There is a quasiconformal map $\hat{f}:\riem \rightarrow \riem_1$ in $\qco(\riem,\riem_1)$
 such that $\hat{f}$ is homotopic
 to $f$ and
 \[  \left. \hat{f} \right|_{\partial_i \riem} = \phi_i.  \]
 The homotopy $G(t,z)$ can be chosen so that for each $t$, $G(t,\cdot)$ is a quasiconformal
 map.
 \end{theorem}
 \begin{proof}
  For each $i$ choose collar charts $(\zeta_i,U_i)$ of $\partial_i \riem$ onto $\mathbb{A}_{r_i}$, say, and
  $(\eta_i,V_i)$ of $\partial_i \riem$ onto $\mathbb{A}_{s_i}$.  For each $i$, choose numbers $R_i \in (1,r_i)$
  and $S_i \in (1,s_i)$.  Let $\gamma_i= \zeta_{i}^{-1}(|z|=R_i)$ and $\beta_i = f(\gamma_i)$.
  The map $\eta_i \circ f \circ \zeta^{-1}_i$ is a quasiconformal mapping from $\mathbb{A}_{R_i}$ onto a
  double connected domain $A$, whose inner boundary is $\{z\,:\,|z|=1\}$ and whose
  outer boundary is the quasicircle $\eta_i(\beta_i)$.
  By \cite[Theorem 2.13(2)]{RadnellSchippers_monster}, the restriction of $\eta_i \circ f \circ \zeta^{-1}_i$ to $|z|=R_i$ is a quasisymmetry (in the sense of Remark \ref{re:quasisymmetry_def}) onto $\eta_i(\beta_i)$.

  By \cite[Corollary 4.1]{RadnellSchippers_monster} there is a quasiconformal mapping from $A$ onto
  itself which is the identity on $\eta_i(\beta_i)$ and equals $\eta_i \circ \phi_i \circ f^{-1} \circ \eta_i^{-1}$ on $|z|=1$.
  In fact, the proof of \cite[Lemma 4.2 and Corollary 4.1]{RadnellSchippers_monster} shows that this map can be embedded
  in a holomorphic motion $h_i: \Delta \times \overline{A} \rightarrow \overline{A}$ (in particular,
  a homotopy of quasiconformal maps) such that $h_i(1,z)=\eta_i \circ \phi_i \circ f^{-1} \circ \eta_i^{-1}$
  for $|z|=1$, $h_i(0,z)=z$ for all $z \in \eta_i(\beta_i)$, and $h_i(t,z)=z$ for $(t,z) \in [0,1] \times \eta_i(\beta_i)$.
  Setting $H_i(t,z) = h(t, \eta_i \circ f_i \circ \zeta_i^{-1}(z))$ we have a homotopy $H_i:[0,1] \times
  \overline{\mathbb{A}_{R_i}} \rightarrow \overline{A}$ such that
  \begin{enumerate}
   \item for each $t \in [0,1]$, $H_i(t,\cdot)$ is a quasiconformal homeomorphism,
   \item $H_i(0,z) = \eta_i \circ f \circ \zeta_i^{-1}(z)$ for all $z$,
   \item $H_i(1,z) = \eta_i \circ \phi_i \circ \zeta_i^{-1}(z)$ for $|z|=1$, and
   \item $H(t,z) = \eta_i \circ f \circ \zeta_i^{-1}(z)$ for $|z|=R_i$.
  \end{enumerate}

  Next, we lift these homotopies back to the surfaces $\riem$ and $\riem_1$
  by composing with the charts, and sew them to $f$.
  Explicitly, we define the map $G:[0,1] \times \riem \rightarrow \riem_1$ by
  \[  G(t,q) =
  \begin{cases} \eta_i^{-1} \circ H(t, \zeta_i(q)),  & \  q \in \zeta_i(A) \\
     f, & \   \text{otherwise}.
    \end{cases}
      \]
  One can verify that $G$ is continuous on the seams $\gamma_i$ by chasing compositions.
  By removability of quasicircles (see \cite[Theorem 3]{Strebel}) $G(t, \cdot)$
  is in fact a quasiconformal homeomorphism for each $t \in [0,1]$
  and in particular a homotopy.  Furthermore $G(0,z)=f(z)$ for all $z \in \riem$ and $G(1,z)=\phi_i(z)$ for
  all $z \in \partial_i \riem$ and $i=1,\ldots, n$.  This concludes the proof.
 \end{proof}
 We say that two quasiconformal maps $f:\riem \rightarrow \riem_1$ and $\hat{f}:\riem \rightarrow \riem_1$
 are homotopic rel boundary if there is a homotopy from $f$ to $\hat{f}$ which is constant on $\partial \riem$.
 (In particular, $f$ and $\hat{f}$ are equal on $\partial \riem$.)
 \begin{theorem} \label{th:qcr_in_Teich_equivalence_class} Let $f: \riem \rightarrow \riem_1$.  $f \in \qco(\riem,\riem_1)$ if and only
 if $f$ is homotopic rel boundary to some $\hat{f} \in \qc_r(\riem,\riem_1)$.
 \end{theorem}
 \begin{proof}
  Let $(\zeta_i,U_i)$ be a collar chart of each boundary curve $\partial_i \riem$ onto $\mathbb{A}_{r_i}$, and
  similarly $(\eta_i,V_i)$ a collar chart of each boundary curve $\partial_i \riem_1$ onto $\mathbb{A}_{s_i}$.
  We may arrange that $U_i \cap U_j$ is empty for $i \neq j$ and similarly for the sets $V_i$.
  Let $\phi_i$ be the restriction of $f$ to $\partial_i \riem$.
  By Theorem \ref{th:GuoHui}, there is a quasiconformal extension $\psi_i$ of $\eta_i \circ \phi_i \circ \zeta_i^{-1}$
  to $\disk^*$ such that its Beltrami differential $\mu(\psi_i) \in \bd(\disk^*)$.  The idea of the rest of the
  proof is to ``patch'' $\eta_i^{-1} \circ \psi \circ \zeta_i$ together with $f$ (note it agrees with $f$ on $\partial_i \riem$);
  since $\mu(\psi_i) \in \bd(\disk^*)$ the resulting map will be in $\qc_r(\riem,\riem_1)$ by definition.

  We now proceed with the patching argument.
  Let $T_i$ be a real number such that $1< T_i <r_i$ and small enough that $f \circ \zeta_{i}^{-1}(|z|=T_i)$ is in $V_i$.
  Since $\psi_i$ is a quasiconformal homeomorphism, there exists an $R_i$ such that $1 < R_i <T_i$ and for
  all $|z|=R_i$
  \[  1 < \psi_i(z) < \min_{w \in \eta_i \circ f \circ \zeta_i^{-1}( |z|=T_i)} |w|.  \]
  In other words, the quasicircles $|w|=1$, $\psi_i(|z|=R_i)$ and $\eta_i \circ f \circ \zeta_i^{-1}(|z|=T_i)$
  are concentric.  Let $B_i$ denote the domain bounded by the latter two curves.  Denote by $A_i$ the annulus $R_i <|z| <T_i$.
  By \cite[Corollary 4.1]{RadnellSchippers_monster}, there is a quasiconformal map $h:A_i \rightarrow B_i$ such
  that $h(z)= \psi_i(z)$ for $|z|=R_i$ and $h(z)=\eta_i \circ f \circ \zeta_i^{-1}(z)$ for $|z|=T_i$.  Let
  \[  \tilde{h}(z)=
  \begin{cases}
   \psi_i(z), \  & z \in \mathbb{A}_{R_i} \\
     h(z), \  & z \in A_i.
     \end{cases} \]
  By removability of quasicircles \cite[Theorem 3]{Strebel}, this extends to a quasiconformal map of
  $\mathbb{A}_{T_i}$ onto $\eta_i \circ f \circ \zeta^{-1}_i(\mathbb{A}_{T_i})$.

  Since $\tilde{h}$ agrees with $\eta_i \circ f_i \circ \zeta_i^{-1}$ on the boundary of the annulus $\mathbb{A}_{T_i}$,
  we have that they are homotopic rel boundary up to a $\mathbb{Z}$ action.  Thus by
  composing with a quasiconformal map $g:A_i \rightarrow A_i$ such that
  $g$ is the identity on $\partial \mathbb{A}_{T_i}$ we can arrange that $\tilde{h}$ is homotopic rel boundary
  to $\eta_i \circ f \circ \zeta_i^{-1}$.  Let $H:[0,1] \times \mathbb{A}_{T_i} \rightarrow \eta_i \circ f \circ \zeta_i^{-1}(\mathbb{A}_{T_i})$
  be such a homotopy.  The important properties of $H$ are that
  \[  \begin{array}{ccll} H(0,z) & = & \eta_i \circ f \circ \zeta_i^{-1}(z), & \ \ z \in \mathbb{A}_{T_i} \\
      H(1,z) & = & \psi_i(z), & \ \ z \in \mathbb{A}_{R_i} \\
      H(t,z) & = & \eta_i \circ f \circ \zeta_i^{-1}(z), & \ \ |z| =T_i, \ t \in [0,1]. \end{array} \]

  Now we define $G:[0,1] \times \riem \rightarrow \riem_1$ by
  \[  G(t,z) =
  \begin{cases}
  \eta_i^{-1} \circ H_i(t,\zeta_i(z)), \  & z \in \zeta^{-1}(\mathbb{A}_{T_i}) \\
       f(z), \ & \text{otherwise}.
  \end{cases}
        \]
  This extends to a quasiconformal map from $\riem$ to $\riem_1$ by removability of
  quasicircles. Chasing compositions we see that $G(0,z)=f(z)$ and
  $G(1,z)= \eta_i^{-1} \circ \psi_i \circ \zeta_i(z)$.  By construction $\hat{f}(z) : = G(1,z) \in \qc_r(\riem,\riem_1)$.
 \end{proof}
 \begin{remark}  As in the previous theorem, the proof shows that the homotopy can be
  chosen so that for each $t$, $G(t,z)$ is a quasiconformal map.
 \end{remark}
 This also establishes
 \begin{corollary}  \label{co:qcr_in_qco} For  bordered Riemann surfaces $\riem$ and $\riem_1$  of type $(g,n)$
   \[\qc_r(\riem,\riem_1) \subseteq
 \qco(\riem,\riem_1).  \]
 \end{corollary}

 Since any two bordered Riemann surfaces of type $(g,n)$ are quasiconformally equivalent,
 these two theorems taken together imply the following.
 \begin{theorem}\label{thm: hui generalization} Let $\riem$ and $\riem_1$ be bordered Riemann surfaces of type $(g,n)$.
 Let $\phi_i:\partial_i \riem \rightarrow \partial_i \riem_1$ be quasisymmetric maps
 for $i=1,\ldots,n$.  Then  $\phi_i \in \qswp(\partial_i\riem,\partial_i \riem_1)$ for all $i=1,\ldots,n$
 if and only if there is a quasiconformal map $f \in \qc_r(\riem,\riem_1)$ such that
 \[  \left. f\right|_{\partial_i \riem} = \phi_i.  \]
 \end{theorem}
 \begin{proof}   Assume that there is a quasiconformal map $f \in \qc_r(\riem,\riem_1)$ whose restriction
  to each $i$th boundary curve is $\phi_i$.   By Corollary \ref{co:qcr_in_qco} $f \in \qco(\riem,\riem_1)$.
  Thus by Definition \ref{de:qco} $\phi_i \in \qswp(\partial_i \riem, \partial_i \riem_1)$.

  Assume now that $\phi_i \in \qswp(\partial_i \riem, \partial_i \riem_1)$ for $i=1,\ldots,n$.
  Since $\riem$ and $\riem_1$ are both of type $(g,n)$, they are quasiconformally equivalent.
  Thus by Theorem \ref{th:qcr_in_Teich_equivalence_class} there is an $f \in \qc_r(\riem,\riem_1)$ whose restriction
  to $\partial_i \riem$ is $\phi_i$ for each $i$.
 \end{proof}
 \begin{remark} \label{re:no_lift_one}
  Note that it is not the case that the Beltrami differential of a quasiconformal map $f$ in $\qc_r(\riem,\riem_1)$
  has a lift in $L^2_{-1,1}(\disk^*)$ (recall this notation refers to differentials which are $L^2$
  with respect to $\lambda_{\disk}(z) |dz|^2$).  In fact, unless the Beltrami differential is zero almost everywhere,
  and hence $f$ is conformal, the integral
  over any fundamental domain is non-zero.  By the invariance of the lifted differential,
  the $L^2$ norm over $\disk^*$ must be the sum over all fundamental domains of the $L^2$ norm
  on $\riem$, and thus is infinite.  In summary, unless $f$ is conformal, the lifted Beltrami differential
  is not in $L^2(\disk^*)$.  Thus there is no ``lifted'' version of Hui's theorem.
 \end{remark}
\end{subsection}
\begin{subsection}{WP-class Teichm\"uller space and the fiber model}
 In this section, we define a WP-class Teichm\"uller space of bordered Riemann surfaces of type $(g,n)$.\\

 \begin{definition}[WP-class Teichm\"uller space] \label{de:WPclass_Teich_refined_model}
  Let $\riem$ be a bordered Riemann surface of type $(g,n)$.  We define
  an equivalence relation on triples $(\riem,f,\riem')$ for $f:\riem \rightarrow \riem'$
  quasiconformal as follows: $(\riem,f_1,\riem_1) \sim (\riem,f_2,\riem_2)$
  if and only if there is a biholomorphism $\sigma:\riem_1 \rightarrow \riem_2$ such that
  $f_2^{-1} \circ \sigma \circ f_1$ is homotopic to the identity rel boundary.

  The WP-class Teichm\"uller space of $\riem$ \cite{RSS_Hilbert} is the set
  \[  \twp(\riem) = \{ (\riem,f,\riem_1) \,:\, f \in \qco(\riem,\riem_1) \} / \sim.  \]
 \end{definition}

In light of Theorem \ref{thm: hui generalization}, we may restrict our attention
 to maps $f \in \qcr(\riem, \riem_1)$.

 \begin{theorem} \label{th:Teich_WP_BD_model} Let $\riem$ be a bordered Riemann surface of type $(g,n)$.
  \[  \twp(\riem) = \{ (\riem,f,\riem_1) \,:\, \mu(f) \in \bd(\riem) \} / \sim  \]
  where $\mu(f)$ denotes the Beltrami differential of $f$.
 \end{theorem}
 \begin{proof}
  By Theorem \ref{th:qcr_in_Teich_equivalence_class}, every equivalence class $[\riem,f,\riem_1]$ in $\twp(\riem)$ has a
  representative of the form $(\riem,g,\riem_1)$ where $g \in \qcr(\riem,\riem_1)$.  By Theorem \ref{th:QSR_is_L2Hyp}
  $\mu(g) \in \bd(\riem)$. On the other hand, if $\mu(g) \in \bd(\riem)$, by Theorem \ref{th:QSR_is_L2Hyp} $g$ is in $\qcr(\riem,\riem_1)$ and thus in $\qco(\riem,\riem_1)$ by Corollary \ref{co:qcr_in_qco}.
 \end{proof}

 \begin{remark}   In \cite[Proposition 2.20]{RSS_Hilbert} we showed that $\qco$ is closed
 under composition.  That is, if $f \in \qco(\riem_1,\riem_2)$ and $g \in \qco(\riem_2,\riem_3)$
  then $g \circ f \in \qco(\riem_1,\riem_3)$.   It would be satisfactory if this were also true
  for $\qcr$, although we do not yet see any independent reason why this should be true. Note that we do not need that result in this paper.
 \end{remark}

 In \cite{RSS_Hilbert}, the authors showed that the WP-class Teichm\"uller space
 is a complex Hilbert manifold.   (The terminology ``WP-class'' was not used there).   The complex structure
 is induced by a fiber model of the Teichm\"uller space \cite{RadnellSchippers_fiber}.  This
 fiber model is a consequence of a correspondence between the rigged moduli space of D. Friedan and S. Shenker
 \cite{FriedanShenker} and C. Vafa \cite{Vafa} appearing in conformal field theory, and the Teichm\"uller space of bordered
 surfaces.
 Therefore in order to define the complex Hilbert manifold structure on $\twp(\riem)$, we must first define
 the so-called rigged moduli space in the WP-category.
 \begin{remark}
 The aforementioned correspondence was discovered by the first two authors
 and first formulated in terms of the classical quasiconformal Teichm\"uller space \cite{RadnellSchippers_monster}.
 The problem of refining the Teichm\"uller space and the rigged moduli space appears to be of interest for both
 the rigorous construction of conformal field theory and for Teichm\"uller theory.
 \end{remark}

 We say that $\riem^P$ is a punctured Riemann surface of type $(g,n)$ if it is a genus $g$ Riemann
 surfaces with $n$ points removed.  The removed points are furthermore given a specific order $p_1,\ldots,p_n$,
 so the term ``punctured Riemann surface of type $(g,n)$'' refers to such a surface together with the ordering.
 For the purposes of this paper we could also think of the
 surfaces as compact Riemann surfaces with $n$ distinguished points listed in a specific order; we will move between
 these two pictures freely without changing the notation.

 \begin{definition}[WP-class riggings]  Let $\riem^P$ be a punctured Riemann surface of type $(g,n)$ with punctures $p_1,\ldots,p_n$.
  A WP-class rigging on $\riem^P$ is an $n$-tuple of maps $\phi=(\phi_1,\ldots,\phi_n)$ such that for each $i=1,\ldots,n$,
  \begin{enumerate}
   \item $\phi_i:\mathbb{D} \rightarrow \riem^P$ is one-to-one and holomorphic,
   \item $\phi_i(0)=p_i$,
   \item $\phi_i$ has a quasiconformal extension to a neighbourhood of the closure $\overline{\disk}$ of $\disk$,
   \item if $\zeta_i:U_i \rightarrow \mathbb{C}$ is a local biholomorphic coordinate on an open set
   $U_i \subset \riem^P$ containing $f_i(\overline{\disk})$ such that $\zeta_i(p_i)=0,$ then
   $\zeta_i \circ \phi_i \in \Oqcwp$, and
   \item Whenever $i \neq j$, $\phi_i(\overline{\disk}) \cap \phi_j(\overline{\disk})$ is empty.
  \end{enumerate}
  Denote the set of WP-class riggings on $\riem^P$ by $\Oqcwp(\riem^P)$.
 \end{definition}

 When $\riem^P$ and $\riem_1^P$ are Riemann surfaces of type $(g,n)$, then if
 $f:\riem^P \rightarrow \riem_1^P$ is quasiconformal, it must extend to the compactification
 in such a way that it takes punctures to punctures.
 \begin{definition}[WP-class rigged Teichm\"uller space]\label{de:rigged teich}
  Let $\riem^P$ be a punctured Riemann surface of type $(g,n)$.  The WP-class rigged Teichm\"uller space
  of $\riem^P$ is the set
  \[  \ttwp(\riem^P) = \{ (\riem^P,f,\riem^P_1,\phi) \}/\sim  \]
  of equivalence classes $[\riem^P,f,\riem^P_1,\phi]$ where
  $f:\riem^P \rightarrow \riem^P_1$ is a quasiconformal map and $\phi \in \Oqcwp(\riem^P_1)$.
  The equivalence relation $\sim$ is defined as follows:
  $(\riem^P,f_1,\riem^P_1,\phi) \sim (\riem^P,f_2,\riem^P_2,\psi)$
  if and only if there is a biholomorphism $\sigma:\riem^P\rightarrow \riem_1^P$ such that
  \begin{enumerate}
  \item  $f_2^{-1} \circ \sigma \circ f_1$ is homotopic to the identity rel
  boundary
  \item $\psi_i = \sigma \circ \phi_i$ for $i=1,\ldots,n$.
  \end{enumerate}
 \end{definition}
 \begin{remark}
 In this case, the requirement that $f_2^{-1} \circ \sigma \circ f_1$ is homotopic to the
 identity rel boundary implies that $f_2^{-1} \circ \sigma \circ f_1(p_i)=p_i$ (in the sense
 that the unique extensions of $f_1$, $f_2$ and $\sigma$ satisfy this equality) and that the punctures are preserved
 throughout the homotopy.
 \end{remark}
 \begin{remark}
  In particular, if $(\riem^P,f_1,\riem^P_1,\phi) \sim (\riem^P,f_2,\riem^P_2,\psi)$ in $\widetilde{T}_0(\riem^P)$
  then $(\riem^P,f_1,\riem^P_1) \sim (\riem^P,f_2,\riem^P_2)$ in the Teichm\"uller space $T(\riem^P)$, and
  $\sigma:\riem^P_1 \rightarrow \riem^P_2$ preserves the ordering of the punctures.
 \end{remark}
\begin{remark} In \cite{RSS_Hilbert} we showed that $\ttwp(\riem^P)$ has a complex Hilbert
 manifold structure in the case that $2g-2+n>0$.  In this paper we will only consider $\ttwp(\riem^P)$
 for $2g-2+n>0$.
 \end{remark}
 Next, we show that the WP-class Teichm\"uller space of bordered surfaces
 and the WP-class rigged Teichm\"uller space of punctured Riemann surfaces are closely related.
 The relation is obtained by ``sewing caps'' onto a given bordered Riemann surface to obtained
 a punctured surface.  We outline this below.

\begin{definition}
 \label{de:analytic param}
 A parametrization $\tau_i : \mathbb{S}^1 \to \partial_i \riem$ of the $i$th boundary curve of a bordered Riemann surface is called analytic if there exists a collar chart $\zeta_i$ such
  that $\zeta_i \circ \tau_i$ is an analytic diffeomorphism of $\mathbb{S}^1$.
 \end{definition}

 Let $\riem$ be a bordered Riemann surface of type $(g,n)$.  Let $\tau=(\tau_1,\ldots,\tau_n)$
 be a WP-class quasisymmetric parameterization of the boundary; that is
 \[   \tau_i \in \qswp(\mathbb{S}^1,\partial_i \riem), \quad i=1,\ldots,n.         \]
 The existence of such a $\tau$ is not restrictive, as the following theorem shows.
 In fact, we show that every boundary curve can be analytically parameterized.
 \begin{theorem} \label{th:WPclass_param_exists}
  There is a parametrization
  $\tau=(\tau_1,\ldots,\tau_n)$ of $\partial \riem$ such that $\tau_i$
  is analytic for each $i$. In particular, $\tau \in \qswp(\mathbb{S}^1,\partial \riem)$.
 \end{theorem}
 \begin{proof}  Let $(\zeta_i,U_i)$ be a collar chart of $\partial_i \riem$.  By the Schwarz
  reflection principle, $\zeta_i$ extends to an open neighbourhood of the boundary $\partial_i \riem$
  in the double $\riem^D$ of the Riemann surface $\riem$, which maps $\partial_i \riem$ into
  the circle $\mathbb{S}^1$.  Let $\tau_i$ be the restriction to $\mathbb{S}^1$ of the extension
  of $\zeta_i^{-1}$.  Then $\zeta_i \circ \tau_i(z)=z$ satisfies the requirement.
 \end{proof}

 Given such a $\tau$, the $\tau_i$'s are in particular quasisymmetries.  Thus one can form the Riemann
 surface
 $\riem^P = \riem \#_\tau \cup_{i=1}^n \disk$ which is defined as follows.
 We identify points $p$ on the boundary of the $n$th disk with points $q$ on the $n$th boundary curves, $p \sim q$,
 if $q=\tau_i(p)$.  The resulting set
 \[  \riem^P = (\riem \sqcup \mathbb{D} \sqcup \cdots \sqcup \mathbb{D})/\sim  \]
 is a topological space with the quotient topology and has a unique complex structure which
 agrees with the complex structures on $\riem$ and each of the discs $\mathbb{D}$ \cite[Theorems 3.2, 3.3]{RadnellSchippers_monster}.
 We refer to $\riem^P$ as being obtained from $\riem$ by ``sewing on caps via $\tau$''.

 The inclusion maps from each disk $\disk$ into $\riem^P$ are holomorphic.
  Denote the resulting maps by $\tilde{\tau}_i:\disk \rightarrow \riem^P$
 and $\tilde{\tau}=(\tilde{\tau}_1,\ldots \tilde{\tau}_n)$.  By \cite[Proposition 5.7]{RSS_Hilbert} $\tilde{\tau} \in \Oqcwp(\riem^P)$.

 Now let $f \in \qco(\riem,\riem_1)$.  Let $\riem_1^P$ be obtained from $\riem_1$ by sewing on
 caps via $f \circ \tau = (f \circ \tau_1,\ldots,f \circ \tau_n)$.  Note that
 $f \circ \tau_i \in \qswp(\partial_i \riem,\partial_i \riem_1)$ for each $i$ by \cite[Proposition 5.8]{RSS_Hilbert}.
 There is a natural extension of $f$ to a map $\tilde{f}:\riem^P \rightarrow \riem_1^P$ defined as follows:
 \begin{equation} \label{eq:rigging_extension_defn}
  \tilde{f}(z) =
  \begin{cases}
  f(z), \  & z \in \overline{\riem} \\ z, \  & z \in \disk \cup \cdots \cup \disk.
  \end{cases}
 \end{equation}

 We thus have a natural map from $\twp(\riem)$ into $\ttwp(\riem^P)$.
 \begin{align} \label{eq:PI_definition}
  \Pi: \twp(\riem) & \longrightarrow \ttwp(\riem^P)  \\\nonumber
  (\riem,f,\riem_1) & \longmapsto (\riem^P,\tilde{f},\riem_1^P,\tilde{f}\circ \tilde{\tau}).
 \end{align}
 Note that $\tilde{f} \circ \tilde{\tau}_i$ is holomorphic on each copy of the disk $\disk$.

 The map $\Pi$ was used in \cite{RSS_Hilbert} to construct an atlas on $\twp(\riem)$, making it
 a complex Hilbert manifold.  We
 will describe the atlas of charts in the next section.

 \begin{remark} \label{re:Berstrick}
  Observe that $\tilde{f}$ is the unique quasiconformal map (up to composition with a conformal map)
  solving the Beltrami equation on $\riem^P$ with the Beltrami differential
  \[ \tilde{\mu}(z) =
  \begin{cases}
   \mu(f)(z),\  & z \in \riem \\
      0, & z \in \disk \cup \cdots \cup \disk.
  \end{cases}
  \]
  Thus we may see the map $\Pi$ as generalizing the Bers trick associating a conformal map $f$
  of the disk with elements of the universal Teichm\"uller space.
 \end{remark}

 Finally we will need the following lemma.
 \begin{lemma} \label{le:analytic_identity}
  Let $\tau \in \qswp(\mathbb{S}^1,\partial \riem)$ be an analytic parametrization of $\partial \riem$.
  Let $\riem^P$ be obtained from $\riem$ by sewing on caps via $\tau$, and let $\tilde{\tau}:\disk \rightarrow \riem^P$ be
  the resulting extension of $\tau$.  For each $i=1,\ldots,n$ there is a chart $(\zeta_i,U_i)$ such that $U_i$
  contains the closure of $\tilde{\tau}_i(\disk)$ and $\zeta_i \circ \tilde{\tau}_i$ is the identity map on $\mathbb{S}^1$.
 \end{lemma}
 \begin{proof}
  Let $(\xi_i,W_i)$ be any coordinate chart such that $\overline{\tilde{\tau}_i(\disk)} \subset W_i$.
  By assumption $\xi_i \circ \tilde{\tau}_i(\mathbb{S}^1)$ is an analytic curve in $\mathbb{C}$,
  so the Riemann map $\eta_i: \disk \rightarrow \xi_i(\tilde{\tau}_i(\disk))$ has a one-to-one
  holomorphic
  extension to some disk $\mathbb{D}_R = \{ z\,:\, |z| <R \}$ for $R >1$ onto a neighbourhood
  of the closure of $\xi_i(\tilde{\tau}_i(\disk))$.  By composing $\eta_i$ by a M\"obius transformation
  we can assume that $\eta_i^{-1} \circ \xi_i \circ \tilde{\tau}_i$ is the identity (since it
  maps the disk to the disk).  Set
  $\zeta_i = \eta_i^{-1} \circ \xi_i$; this is a chart on some domain $U_i \subseteq W_i$
  where $U_i$ contains the closure of $\tilde{\tau}_i(\disk)$.
 \end{proof}
\end{subsection}

%%%%%%%%%%%% aaa %%%%%%%%%%%%
\begin{subsection}{Preliminaries on marked holomorphic families}

In this section we collect some standard definitions and facts about marked holomorphic families of Riemann surfaces and the universality of the \teich curve. These will play a key role in the construction of the complex structure on the WP-class \teich space in Section \ref{se:complexstucture}. A full treatment appears in \cite{EarleFowler}, and also in the books \cite{Hubbard, Nagbook}.

\begin{definition} \label{de:holo_family}
A holomorphic family of complex manifolds is a pair of connected complex manifolds $(E,B)$ together with a surjective holomorphic map  $\pi : E \to B$ such that (1) $\pi$ is topologically a locally trivial fiber bundle, and (2) $\pi$ is a split submersion (that is, the derivative is a surjective map whose kernel is a direct summand).
\end{definition}

\begin{definition}  A \textit{morphism of holomorphic families} from $(E',B')$ and $(E,B)$ is a pair of holomorphic maps $(\alpha, \beta)$ with $\alpha : B' \to B$ and $\beta: E' \to E$ such that
$$
\xymatrix{
E' \ar[r]^{\beta} \ar[d]_{\pi'} & E \ar[d]^{\pi} \\
B' \ar[r]^{\alpha} & B
}
$$
commutes, and for each fixed $t\in B'$, the restriction of $\beta$ to the
fiber $\pi'^{-1}(t)$   is a biholomorphism onto $\pi^{-1}(\alpha(t))$.
\end{definition}

Throughout, $(E,B)$ will be a holomorphic family of Riemann surfaces; that is, each fiber $\pi^{-1}(t)$ is a Riemann surface.

%Moreover, since our trivialization will always be global we specialize the standard definitions (see \cite{EarleFowler}) to this case in what follows.

Let $\riem$ be a punctured Riemann surface of type $(g,n)$.  This fixed surface $\riem$ will serve as a model of the fiber.
Let $U$ be an open subset of $B$.

\begin{definition} \label{de:strong_triv}\text{}
 \begin{enumerate}
\item  A  \textit{local trivialization} of $\pi^{-1}(U)$ is a homeomorphism $\theta: U \times \riem \to E$ such that $\pi(\theta(t,x)) = t$ for all $(t,x) \in U\times \riem$.
\item A local trivialization $\theta$ is a \textit{strong local trivialization} if for fixed $x \in \riem$, $t\mapsto \theta(t,x)$ is holomorphic, and for each $t\in U$, $x \mapsto \theta(t,x)$ is a quasiconformal map from $\riem$ onto $\pi^{-1}(t)$.
\item $\theta : U\times \riem \to E$ and $\theta' : U\times \riem \to E$ are \textit{compatible} if and only if $\theta'(t,x) = \theta(t,\phi(t,x))$ where for each fixed $t$, $\phi(t,x) : \riem \to \riem$ is a quasiconformal homeomorphism that is homotopic to the identity rel boundary.
\item A \textit{marking} $\mathcal{M}$ for $\pi : E \to B$ is a set of equivalence classes of compatible strong local trivializations that cover $B$.
\item A \textit{marked holomorphic family of Riemann surfaces} is a holomorphic family of Riemann surfaces with a specified marking.
\end{enumerate}
\end{definition}

\begin{remark} Let $\theta$ and $\theta'$ be compatible strong local trivializations. For each fixed $t\in U$, $[\riem, \theta(t,\cdot), \pi^{-1}(t)] = [\riem, \theta'(t,\cdot), \pi^{-1}(t)]$ in $T(\riem)$, so a marking specifies a \teich equivalence class for each $t$.
\end{remark}

We now define the equivalence of marked families.

\begin{definition}
A \textit{morphism of marked holomorphic families} from $\pi':E' \to B'$ to $\pi: E \to B$ is a pair of holomorphic maps $(\alpha,\beta)$ with $\beta : E' \to E$ and $\alpha: B' \to B$ such that
\begin{enumerate}
\item $(\alpha,\beta)$ is a morphism of holomorphic families, and
\item the markings $B' \times \riem \to E$ given by  $\beta(\theta'(t,x))$ and $\theta(\alpha(t), x)$ are compatible.
\end{enumerate}
\end{definition}
The second condition says that $(\alpha,\beta)$ preserves the marking.

\begin{remark}[relation to Teichm\"uller equivalence]\label{rem:relation to teich equivalence}
\label{re:marking_preserving}
Define $E = \{(s,Y_s)\}_{s\in B}$ and $E' = \{(t,X_t)\}_{t\in B'}$ to be marked families of Riemann surfaces with markings $\theta(s, x) = (s, g_s(x))$ and $\theta'(t, x) = (t, f_t(x))$ respectively. Assume that $(\alpha,\beta)$ is a morphism of marked families, and define $\sigma_t$ by
$\beta(t, y) = (\alpha(t), \sigma_t(y))$. Then $\beta(\theta'(t,x)) = (\alpha(t), \sigma_t(f_t(x)))$ and $\theta(\alpha(t), x) = (\alpha(t), g_{\alpha(t)}(x))$. The condition that $(\alpha, \beta)$ is a morphism of marked families is simply that $\sigma_t \circ f_t$ is homotopic rel boundary to $g_{\alpha(t)}$.
That is, when $s = \alpha(t)$,  $[\riem, f_t, X_t] = [\riem, g_s, Y_s]$  via the biholomorphism $\sigma_t : X_t \to Y_s$.
\end{remark}

Let $\mathbb{H}$ be the upper-half plane and $G$ be a Fuchsian group such that $\riem = \mathbb{H}/G$ is a punctured Riemann surface (thus $2g-2+n>0$). Let $T(G)$ be the ``$\mu$-model'' of the \teich space of $\riem$. Let $\mu$ be a Beltrami differential on $\riem$ and let $\tilde{\mu}$ be the lift of $\mu$ to $\mathbb{H}$ and extended by $0$ to the lower half plane. Let $w^{\mu}$ be the normalized solution of the Beltrami equation on $\mathbb{C}$ with dilatation $\tilde{\mu}$. The \textit{Bers fiber space} is the subset $F(G) \subset T(G) \times \mathbb{C}$ given by
$$
F(G) = \{([\mu], z) \st [\mu] \in T(G) , z \in w^{\mu}(\mathbb{H}) \}.
$$
Let $\mathcal{T}(G) = F(G)/G$.
The group $G$ acts freely and properly discontinuously by biholomorphisms on $F(G)$ and the quotient is a marked holomorphic family.
We will define the marking below.
\begin{definition} The marked holomorphic family of Riemann surfaces
\begin{align*}
\pi_T : \mathcal{T}(G) & \to T(G) \\
([\mu], z) & \mapsto [\mu]
\end{align*}
is called the \textit{\teich curve}.
\end{definition}

Let $G^\mu = w^{\mu} \circ G \circ (w^{\mu})^{-1}$. The fiber above a point $[\mu] \in T(G)$ is the canonical Riemann surface
\begin{equation}
\label{eq:canonicalrep}
\riem_{\mu} = w^{\mu}(\mathbb{H}) / G^{\mu}
\end{equation}
which is independent of the \teich equivalence class representative. The map $w^{\mu}$ uniquely defines a map
\begin{equation}
\label{eq:f_mu}
f_{\mu}: \riem \to \riem_{\mu} .
\end{equation}
Note that while the boundary values of these maps are independent of the \teich equivalence class representative, the maps themselves are not.
\begin{remark}
The use of the subscript $\mu$ on $f$ and $\riem$ instead of a superscript is for notational simplicity later in the document, and it should not be confused with a different meaning of the lower subscript common in \teich literature.
\end{remark}

Let $L^{\infty}_{-1,1}(\riem)_1$ denote the open unit ball in $L^{\infty}_{-1,1}(\riem)$ and let $U$ be an open subset of $T(\riem)$ for which a holomorphic section of the fundamental projection $L^{\infty}_{-1,1}(\riem)_1 \to T(\riem)$ exists.
The strong local trivialization
$$
\theta : U \time \riem \to \pi_T^{-1}(U)
$$
is defined by $\theta([\mu], z) = f_{\mu}(z)$

The following universal property
of $\mathcal{T}(\riem)$ (see \cite{EarleFowler, Hubbard, Nagbook}) is all that we need for our purposes.
\begin{theorem}[Universality of the \teich curve]
\label{th:UniversalProperty}
Let $\pi : E \to B$ be a marked holomorphic family of Riemann surfaces with fibre model $\riem$ of type $(g,n)$ with $2g-2+n>0$, and trivialization $\theta$.
Then there exists a unique morphism $(\alpha, \beta)$ of marked families from $\pi : E \to B$ to $\pi_T : \mathcal{T}(\riem) \to T(\riem)$. Moreover, the canonical ``classifying" map $\alpha: B \to T(\riem)$ is given by $\alpha(t) = [\riem, \theta(t,\cdot), \pi^{-1}(t)]$.
\end{theorem}

\end{subsection}

\begin{subsection}{Gardiner-Schiffer coordinates and the complex structure}
\label{se:complexstucture}
 In this section we define Gardiner-Schiffer coordinates and the complex structure on $\twp(\riem^B)$.
 Although the geometric idea is straightforward, the construction and rigorous proofs are somewhat involved.
 We restrict ourselves here to summarizing the necessary
 facts and refer the reader to \cite{RSS_Hilbert} for a full treatment.

 Gardiner-Schiffer variation is a technique for constructing coordinates on the
 Teichm\"uller space of a compact surface with punctures.

 Let $\riem^P$ be a punctured Riemann surface.  Let $D=(D_1,\ldots,D_d)$ be an $n$-tuple
 of disjoint open sets $D_i \subset \riem^P$ each of which is biholomorphic to the
 unit disk $\disk$.  Here $d$ is the
 dimension of the Teichm\"uller space $T(\riem^P)$ (note that we are treating the Riemann
 surface $\riem^P$ as a punctured surface of type $(g,n)$ and hence the dimension is $3g-3+n$).
 Let $\epsilon=(\epsilon_1,\ldots,\epsilon_d) \in \mathbb{C}^d$ be close to $(0,\ldots,0)$.
 By excising the disks $D_i$, distorting them under the quasiconformal map $z \mapsto z + \epsilon_i \bar{z}$,
 and sewing them back in, one obtains a new Riemann surface $\riem_{\epsilon}^P$ and
 a quasiconformal map $\nu_\epsilon :\riem^P \rightarrow \riem_\epsilon^P$ defined by $\nu_{\epsilon}(z) = z + \epsilon_i \bar{z}$ for $z\in D_i$ and  $\nu_{\epsilon}(z) = z$ for $z \in \riem^P \setminus D$. Note that $\nu_{\epsilon}$ is holomorphic outside of the disks $D_i$.  We can use this to obtain coordinates
 on $T(\riem^P)$ as follows.
 \begin{theorem} \label{th:Schiffer_variation}
  Let $\riem^P$ be a punctured Riemann surface of type $(g,n)$.  Let $[\riem^P,f,\riem_1^P] \in
  T(\riem^P)$. Let $D_i \subset \riem_1^P, i=1,\ldots, d$ be any disjoint biholomorphic images of the unit disk $\disk$. There exist coordinates $\eta_i: D_i \to \disk$ for
  $i=1,\ldots,d$, and an open connected neighbourhood $\Omega$ of $0 \in \mathbb{C}^d$, such that
  \begin{enumerate}
 \item For all $\epsilon \in \Omega$, $\nu_\epsilon:\riem_1^P \rightarrow \riem_\epsilon^P$ is holomorphic on $\riem_1^P \backslash \overline{\cup_{i=1}^d D_i}$.
   The Beltrami differential of $\nu_\epsilon$ is $\epsilon_i d\bar{z}/dz$ in the coordinate system $\eta_i$ on $D_i$.
   \item The map
   \begin{align*}
    \mathfrak{S}:\Omega & \longrightarrow T(\riem^P) \\
    \epsilon & \longmapsto [\riem^P,\nu_\epsilon \circ f, \riem_\epsilon^P]
   \end{align*}
   is a biholomorphism onto its image $($i.e. the inverse is a local coordinate chart$)$.
  \end{enumerate}

 \end{theorem}
A proof can be found in \cite{Gardiner} or \cite{Nagbook}.

 The collection of Riemann surfaces $\riem_\epsilon^P$ form a marked holomorphic family of Riemann surfaces over
 $\Omega$ \cite[Section 4.4]{RSS_Hilbert}.  That is, the set
 $$
 S(\Omega,D) = \{ (\epsilon, x)\,:\, \epsilon \in \Omega \text{ and }  x \in \riem_\epsilon^P \}
 $$
 is a complex manifold, and the projection $\pi : S(\Omega,D) \rightarrow \Omega$ satisfies
 \begin{enumerate}
  \item $\pi$ is holomorphic,
  \item For any $\epsilon$, $\pi^{-1}(\epsilon)=\riem_\epsilon^P$, and
  \item $\riem_\epsilon^P$ is a complex submanifold of $S(\Omega,D)$.
 \end{enumerate}
 Furthermore, the map
 \begin{align} \label{eq:strong_trivialization}
  \theta: \Omega \times \riem^P & \rightarrow S(\Omega,D) \\
  (\epsilon,q) & \mapsto (\epsilon,\nu_\epsilon \circ f(q)) \nonumber
 \end{align}
 is a ``strong global trivialization''.  That is,
 \begin{enumerate} \addtocounter{enumi}{3}
  \item for fixed $x$, $\epsilon \mapsto \theta(\epsilon,x)$ is holomorphic, and
  \item for each fixed $\epsilon$, $x \mapsto \theta(\epsilon,x)$ is a quasiconformal map
  from $\riem$ to $\pi^{-1}(\epsilon) = \riem_\epsilon^P$.
 \end{enumerate}

By \cite[Theorem 4.15]{RSS_Hilbert} $\pi:S(\Omega,D) \rightarrow \Omega$ is a marked holomorphic family
 with marking $\theta$ (see Definition \ref{de:strong_triv} ).

Using Theorem \ref{th:UniversalProperty} we can embed the marked Schiffer family into the universal Teichm\"uller curve as described in the following theorem.
Recall that $\mathfrak{S}(\epsilon) = [\riem^P, \nu^{\epsilon} \circ f, \riem_{\epsilon}^P]$.
Let $\mu(\epsilon) = \mu(\nu^{\epsilon} \circ f)$. Then the fiber in the \teich curve over $\mathfrak{S}(\epsilon)$ is the canonical Riemann surface $\riem_{\mu(\epsilon)}^P$ as defined in (\ref{eq:canonicalrep}). Note that $\riem_{\epsilon}^P$ and $\riem_{\mu(\epsilon)}^P$ are biholomorphically equivalent.

 \begin{theorem} \label{th:Schiffer_to_Teich_curve}  Let $\riem^P$ be a punctured Riemann surface of type $(g,n)$
  such that $2g - 2 + n >0$.  Let $S(\Omega,D)$ be the marked Schiffer family corresponding
  to the Schiffer variation $\mathfrak{S}:\Omega \rightarrow T(\riem^P)$.  There is a morphism of marked holomorphic families
  $(\alpha, \beta)$ from $\pi:S(\Omega,D) \rightarrow \Omega$ to $\pi_T : \mathcal{T}(\riem^P) \to T(\riem^P)$ and moreover
  \begin{enumerate}
  \item $\alpha(\epsilon) = \mathfrak{S}(\epsilon) = [\riem^P,\nu_\epsilon \circ f, \riem_\epsilon^P]$, and
  \item $\alpha$ and $\beta$ are injective and holomorphic.
  \end{enumerate}
  Setting
  $$
  \sigma_{\epsilon}(z) = \beta(\epsilon, z) : \riem_{\epsilon}^P \to \riem_{\mu(\epsilon)}^P
  $$
  we have the biholomorphism
 \begin{align*}
  \Gamma: S(\Omega,D) & \longrightarrow \pi_T^{-1}(\mathfrak{S}(\Omega)) \subseteq \mathcal{T}(\riem^P)  \\
  (\epsilon,p) & \longmapsto (\alpha, \beta)(\epsilon, p) = \left( [\riem^P,\nu_\epsilon \circ f, \riem_\epsilon^P], \sigma_\epsilon(p) \right).
 \end{align*}
 \end{theorem}
 \begin{proof}
  The only part not following immediately from Theorem \ref{th:UniversalProperty} is the injectivity of $\alpha$ and $\beta$.
  Since $\mathfrak{S}$ is injective by Theorem \ref{th:Schiffer_variation}, it remains only to show that $\beta$ is injective.
  This follows from the fact that $\beta$ is injective fiberwise and $\alpha \circ \pi = \pi_T \circ \beta$.
 \end{proof}

 Next, we need to define a local trivialization of the elements of $\Oqcwp(\riem^P)$.
 \begin{definition} \label{de:n-chart}
  Let $\riem^P$ be a punctured Riemann surface of type $(g,n)$.  An $n$-chart $(\zeta,E)$ on $\riem^P$
  is a collection of open subsets $E=(E_1,\ldots,E_n)$ of the compactification of $\riem^P$ and local biholomorphic parameters
  $\zeta_i:E_i \rightarrow \mathbb{C}$  such that
  \begin{enumerate}
   \item $E_i \cap E_j$ is empty whenever $i \neq j$,
   \item $p_i \in E_i$ for $i=1,\ldots,n$, and
   \item $\zeta_i(p_i)=0$ for $i=1,\ldots,n$.
  \end{enumerate}
 \end{definition}
 \begin{definition}
  Let $\riem^P$ be a punctured Riemann surface of type $(g,n)$.  Let $(\zeta,E)$ be
  an $n$-chart on $\riem^P$.  We say that a collection
  $U_1,\ldots,U_n$ of open sets $U_i \subseteq \Oqcwp$ is compatible with $(\zeta,E)$
  if for all $f_i \in U_i$, $\overline{f_i(\disk)} \subseteq \zeta_i(E_i)$.
  In this case we will also say that $U=U_1 \times \cdots \times U_n \subset \Oqcwp \times \cdots \times \Oqcwp$
  is compatible with $(\zeta,E)$.
 \end{definition}
 The existence of such open sets is not immediately obvious.  By \cite[Theorem 3.4]{RSS_Hilbert}, for
 any open subsets $F_i$ of $\zeta_i(E_i)$, $i=1,\ldots,n$,
 \[  \{ f_i \,:\, \overline{f_i(\disk)} \subseteq F_i \} \]
 is open in $\Oqcwp$ and thus for example
 \[  \{ (f_1,\ldots,f_n) \in \Oqcwp \,: \, \overline{f_i(\disk)} \subseteq F_i, \ i=1,\ldots,n \}  \]
 is open in $\Oqcwp \times \cdots \times \Oqcwp$.

 %For $U$ compatible with $(\zeta,E)$ let
 %\[  V_{\zeta,E,U} = \{ g = (g_1,\ldots,g_n) \in \Oqcwp(\riem^P) \,: \, \zeta_i \circ g_i \in U_i \ i=1,\ldots,n \}.  \]

 \begin{definition}
  Let $\mathfrak{S}:\Omega \rightarrow T(\riem^P)$ be Schiffer coordinates based at $[\riem^P,f,\riem_1^P]$.
  We say that an $n$-chart $(\zeta,E)$ is compatible with $\mathfrak{S}$ if the closure of each disk $D_i$, $i=1,\ldots,d$
  in the Schiffer variation is disjoint from the closure of each open set $E_i$ in the $n$-chart.
 \end{definition}
 This definition ensures that the Schiffer variation maps $\nu_\epsilon$ are conformal on the closures of each $n$-chart;
 this is crucial for the construction of coordinates on $\ttwp(\riem)$.

 \begin{definition}\label{defn:fusomega}  Let $\mathfrak{S}:\Omega \rightarrow T(\riem^P)$ be Schiffer coordinates based at $[\riem^P,f,\riem_1^P]$ corresponding
  to disks $D=(D_1,\ldots,D_d)$, and let $S = S(\Omega, D)$ be the corresponding Schiffer family. Let $(\zeta,E)$ be an $n$-chart on $[\riem^P,f,\riem_1^P]$ and assume that $\mathfrak{S}$
  is compatible with an $n$-chart $(\zeta,E)$.  Let $U$ be an open
  subset of $\Oqcwp \times \cdots \times \Oqcwp$ which is compatible with $(\zeta,E)$.

  Define $F(U,S,\Omega)$ by
  \[  F(U,S,\Omega) = \{ (\riem^P,\nu_\epsilon \circ f, \nu_\epsilon (\riem_1^P),\psi) \,:\, \psi = \nu_\epsilon \circ \zeta^{-1} \circ \phi,  \ \phi \in
    U, \ \epsilon \in \Omega \}  \]
    where $\zeta^{-1} \circ \phi = (\zeta^{-1}_1 \circ \phi_1,\ldots,\zeta^{-1}_n \circ \phi_n)$.
 \end{definition}

 It was shown in \cite{RSS_Hilbert} that these define a base for a topology on $\ttwp(\riem^P)$, and furthermore that
 the set of charts
 \begin{align} \label{eq:old_chart_definition}
  \mathcal{G}:F(U,S,\Omega) & \longrightarrow \Omega \times \Oqcwp \times \cdots \Oqcwp \nonumber \\
  (\riem^P,\nu_\epsilon \circ f, \nu_\epsilon (\riem_1^P),\psi) & \longmapsto (\epsilon, \zeta_i \circ \nu_\epsilon^{-1} \circ \psi)
 \end{align}
 is an atlas defining a complex Hilbert manifold structure on $\ttwp(\riem^P)$.
 \begin{remark}  In \cite{RSS_Hilbert}, we used the notation $G$ for $\mathcal{G}^{-1}$.
 \end{remark}

 The complex structure on $\ttwp(\riem^P)$ passes upwards to $\twp(\riem)$,
 according to the following theorem which the authors proved in \cite[Theorem 5.8]{RSS_Hilbert}.
 \begin{theorem} \label{th:Pi_projection_holomorphic} If $\riem$ is a bordered surface of type $(g,n)$, with 
 $2g-2+n>0$, then the collection of charts on
  $\twp(\riem)$ of the form
  $\mathcal{G} \circ \Pi$ is an atlas for a complex structure.  Thus
  $\twp(\riem)$ is a complex Hilbert manifold.  The map $\Pi$ defined by (\ref{eq:PI_definition}) is a local
  biholomorphism; that is, for any point
  $p \in \twp(\riem)$ there is an open neighbourhood $U$ of $p$ such that the restriction of
  $\Pi$ to $U$ is a biholomorphism onto its image.
 \end{theorem}
 We also have a natural projection from $\ttwp(\riem^P)$ onto $T(\riem^P)$, given by
 \begin{align} \label{eq:F_projection_definition}
  \mathcal{F}: \ttwp(\riem^P) & \longrightarrow T(\riem^P) \\ \nonumber
  ([\riem^P,f,\riem_1^P],\psi) & \longmapsto [\riem^P,f,\riem_1^P].
 \end{align}
\end{subsection}
\begin{subsection}{A preparation theorem}
 The purpose of this section is to use the sewing technology to prove the following fact:
 given a holomorphic curve through the identity in $\twp(\riem)$,
 it is possible to choose a Teichm\"uller equivalent curve in $\bd(\riem)$.  That is,
 the curve can be chosen so that at each point the representative of the Teichm\"uller
 equivalence class
 has a Beltrami differential simultaneously in $L^2_{-1,1}(\riem)$ and $L^\infty_{-1,1}(\riem)$.

 To prove this, we first need the following variation on a lemma of Takhtajan and Teo.
 Denote by $\mathcal{S}(\phi)$ the Schwarzian of  $\phi$ and by $\mathcal{A}(\phi)$
 the pre-Schwarzian.
 \begin{lemma} \label{le:TT_variation}
There is an open neighbourhood of $0 \in A_1^2(\disk)$ such that the map
  \begin{align*}
   \Psi:A_1^2(\disk) & \longrightarrow A_2^2(\disk) \oplus \mathbb{C} \\
    \psi & \longmapsto \left( \psi_z - \frac{1}{2} \psi^2 , \psi(0) \right)
  \end{align*}
  is a biholomorphism onto its image.

  In particular, there exists an open ball $A$ centred on $0$ in $\Oqcwp$ on which
  \begin{equation}
\| \mathcal{S}(\phi) \|^2_{2,\disk} + \left| \mathcal{A}(\phi)(0) \right|^2 \approx
     \| \mathcal{A}(\phi) \|^2_{2,\disk}
  \end{equation}
  where we are treating $\mathcal{S}(\phi)$ as a quadratic differential in $A_2^2(\disk)$
  and $\mathcal{A}(\phi)$ as a one-differential in $A_1^2(\disk)$.
 \end{lemma}

 \begin{proof}
 The map $\Psi$ is in fact holomorphic and injective on $A_1^2(\disk)$ by \cite[Lemma A.1]{Takhtajan_Teo_Memoirs}.
  In particular, $\left. D\Psi \right|_0$ is injective and bounded.  We claim that $\left. D\Psi \right|_0$ is
  also surjective.

  We will use the following approach. Let $(h,\alpha) \in A_2^2(\disk) \oplus \mathbb{C}$.  Let $\phi_t$ be the unique
  holomorphic function on $\mathbb{D}$ satisfying $\phi_t(0)=0$, $\phi_t'(0)=1$, $\mathcal{A}(\phi_t)(0)=t\alpha$
   and
  $\mathcal{S}(\phi_t(z))= t h(z)$.
 We will show that for some $\delta >0$, if $|t| <\delta$ then $(\mathcal{S}(\phi_t(z)),\mathcal{A}(\phi_t)(0))$ is in the image of
 $\Psi$.  This will show that $\left. D\Psi \right|_0$ is surjective, and thus
 by the open mapping theorem for bounded linear maps it would follow that
 $\left. D\Psi \right|_0$ is a topological isomorphism.  Thus by the inverse function theorem the inverse
 is holomorphic \cite{Lang}.
Now if $t$ is small enough then $\phi_t(z) \in \Oqc$.  To see this, observe
 first that by holomorphicity of inclusion $A_2^2(\disk) \hookrightarrow A_2^\infty(\disk)$ if
 $\| \mathcal{S}(\phi_t(z)) \|_{2,\disk}$ is small enough then by the classical criterion for
 quasiconformal extendibility, $\phi_t(z)$ has a quasiconformal extension to $\mathbb{C}$.
 By \cite[Theorem 1.12, Part II]{Takhtajan_Teo_Memoirs}, this together with the bound
 on $\| \mathcal{S}(\phi_t) \|_{2,\mathbb{D}}$ shows that $\mathcal{A}(\phi_t) \in A_1^2(\disk)$.
 This completes the proof of the first claim. The final claim follows from the scale invariance of $\mathcal{S}(\phi)$ and $\mathcal{A}(\phi)$.
 \end{proof}

  Now we prove a result which plays a central role in what follows.
  \begin{theorem} \label{thm:local_bounded_one_param_Upp} Let $\riem$ be a bordered 
  Riemann surface of type $(g,n)$ such that $2g-2+n>0$.  Assume that $t \mapsto [\riem,f_t,\riem_t]$ is a holomorphic one-parameter curve
   in $\twp(\riem)$ such that $[\riem,f_0,\riem_0] = [\riem, \text{Id},\riem]$.
   Then there exists $\delta>0$ and representatives $(\riem,f_t,\riem_t)$ so that the following properties are satisfied for $|t|<\delta$:
   \begin{enumerate}
   \item $\| \mu(f_t) \|_2$ is uniformly bounded in $t$, where $\mu(f_t)$ is the Beltrami differential of $f_t$ on $\riem$,
   \item $t \mapsto \mu(f_t)$ is a holomorphic curve in $L^2_{-1,1}(\riem)$, and
   \item $t \mapsto \mu(f_t)$ is a holomorphic curve in $L^\infty_{-1,1}(\riem)$.
   \end{enumerate}
  \end{theorem}
  \begin{proof}
 As before, sew on caps via some $\tau \in \qswp(\riem)$
 to obtain the associated punctured surface $\riem^P$.  By Theorem \ref{th:Pi_projection_holomorphic}
 every curve through $[\riem,\mbox{Id},\riem]$ is the inverse image under
 $\Pi$ of a curve through $[\riem^P,\text{Id},\riem^P,\tilde{\tau}] \in \ttwp(\riem)$.  Thus
 it suffices to describe curves through this point in $\ttwp(\riem)$.  By \cite[Corollary 5.22]{RSS_Hilbert}
 the complex structure on $\twp(\riem)$ is independent of the choice of rigging $\tau$ used to sew on
 caps. Thus without loss of generality, we may assume that $\tau$ is an analytic parametrization (and the
 existence of such a $\tau$ is guaranteed by Theorem \ref{th:WPclass_param_exists}).
 By Lemma \ref{le:analytic_identity} there is an $n$-chart $(\zeta,E)$ such that $\zeta_i \circ \tilde{\tau}_i$
 is the identity on $\disk$.

 Fix a coordinate chart $\mathcal{G}$ on
 a neighbourhood $F(U,S,\Omega)$ of $([\riem^P,\text{Id},\riem^P],\tilde{\tau})$ in $\ttwp(\riem^P)$,
 with associated compatible $n$-chart $(\zeta,E)$ on $\riem^P$.  Let $K_i$ be open, connected,
 simply-connected neighbourhoods of
 each puncture for $i=1,\ldots,n$ such that  $\overline{K_i} \subset E_i$ for $i=1,\ldots,n$. We may choose the chart $F(U,S,\Omega)$ in Definition \ref{defn:fusomega} above such that $\overline{\phi_i(\disk)} \subset \zeta_i(K_i)$ for all $\phi_i
 \in U$ and $i=1,\ldots,n$.

Given the holomorphic curve $\lambda(t)$ in $T_{WP}(\riem)$, assume that it is in the image of
 some chart $\mathcal{G} \circ \Pi$ on $F(U,S,\Omega)$ of the above form (perhaps shrinking the $t$-domain of the curve). Let $(\epsilon(t),w_1(t,z),\ldots,w_n(t,z))\in \Omega \times \Oqcwp \times \cdots \Oqcwp = \mathcal{G} \circ \Pi(\lambda(t))$ with $w_i(0,z)=z$, $i=1,\ldots,n$. By \cite[Lemma 2.1, Page 22]{Takhtajan_Teo_Memoirs}, one has
\begin{equation}\label{tnt estimate}
  \sup_{\mathbb{D}}(1-|z|^2)^2 |\mathcal{S}(w_i(t,\cdot))(z)|\leq \sqrt{\frac{12}{\pi}}\left\{\iint_{\mathbb{D}} (1-|z|^2)^2 |\mathcal{S}(w_i(t,\cdot))(z)|^2\, dA\right\}^{\frac{1}{2}}.
\end{equation}
Now using Lemma \ref{le:TT_variation}, there exists $C>0$ such that for all $i=1,\ldots,n$ one has
\begin{equation}\label{estimating schwarzian in terms of preschwarzian}
\left\{\iint_{\mathbb{D}} (1-|z|^2)^2 |\mathcal{S}(w_i(t,\cdot))(z)|^2\, dA\right\}^{\frac{1}{2}}\leq C (|\mathcal{A}(w_i(t,\cdot))(0)| + \left\{\iint_{\mathbb{D}} |\mathcal{A}(w_i(t,\cdot))(z)|^2\, dA\right\}^{\frac{1}{2}}).
\end{equation}

 Letting $P_i(t):= |\mathcal{A}(w_i(t,\cdot))(0)| + \{\iint_{\mathbb{D}} |\mathcal{A}(w_i(t,\cdot))(z)|^2\, dA\}^{\frac{1}{2}}$ we know that $\mathcal{A}(w_i(t,\cdot))(0)$ is continuous and by the definition of the complex structure on $\twp$, $\Vert \mathcal{A}(w_i(t,\cdot))\Vert_{L^2(\mathbb{D} )}$ is also a continuous function of $t.$ Therefore since $\mathcal{A}(w_i(0,\cdot))=0$ we have $P_i(0)=0$ and hence there exists an $s>0$ such that $P_i(t)\leq \sqrt{\pi/12C^2}$ for all $i=1,\ldots,n$, if $|t|<s$. This fact together with \eqref{tnt estimate} yields

\begin{equation}\label{bound on the schwarzian}
 \sup_{i=1,\ldots,n}\sup_{\mathbb{D}} (1-|z|^2)^2 | \mathcal{S}(w_i(t,\cdot))(z)| \leq 1.
 \end{equation}
 Given \eqref{bound on the schwarzian}, the Ahlfors-Weill reflection result \cite[Theorem 5.1, Chpt II]{Lehtobook}
 yields that $w_i(t,z)$ has a jointly continuous quasiconformal extension to $\mathbb{C}$ with dilatation $m^{i}_t$ satisfying
 \begin{equation} \label{eq:mu_S formula}
 m^{i}_t(1/\bar{z}) = - \frac{(1-|z|^2)^2}{2} \frac{z^2}{\bar{z}^2} \mathcal{S}(w_i(t,\cdot))(z) .
 \end{equation}
 Thus for $|t|<s$
\begin{equation}\label{estimate for mu}
\iint_{\mathbb{D}}\frac{|m^{i}_t(1/\bar{z})|^2}{(1-|z|^2)^2}\, dA = \frac{1}{4} \iint_{\mathbb{D}} (1-|z|^2)^2 |\mathcal{S}(w_i(t,\cdot))(z)|^2\, dA.
\end{equation}
Now using once again \eqref{estimating schwarzian in terms of preschwarzian} and the bound on $P_i(t)$ obtained above, it is readily seen that there is an $M$ such that $\|m^i_t\|^2_{2,\disk^*} \leq M$ for each $i$ and small enough $t.$ This yields the uniform boundedness of $\| m^i_t\|_{2, \disk^*}$ for $|t|<s$.

 Choose simple closed analytic curves $\gamma_i$ in $\zeta_i(E_i)$ so that $\gamma_i$ and $\overline{K_i}$ are nested and non-intersecting; that is $\gamma_i$ encloses $\overline{K_i}$.  Choose $R_i >1$ such that $w_i(t, |z| = R_i)$
 is enclosed by $\gamma_i$ and does not intersect it; this can be done for all $|t|<s'<s$ for some $s'$ since
 the extension $w_i(t,z)$ is jointly continuous.  Let $L_i$ be the doubly-connected region bounded by $\zeta_i^{-1}(|z|=R_i)$ and
 $\zeta_i^{-1}(\gamma_i)$.   Let $Y$ denote the pre-compact subset of $\riem$ bounded by the curves $\zeta_i^{-1}(\gamma_i)$.

We will construct a family of quasiconformal maps $g_t:\riem^P \rightarrow \riem^P$ with the following properties.
 \begin{enumerate}
  \item[(a)] $g_t$ is the identity on $\overline{Y}$.
  \item[(b)] $\zeta_i \circ g_t \circ \zeta_i^{-1}$ is the Ahlfors-Weill extension of $w_i(t, z)$ on
   $|z| \leq R$ with dilatation given by (\ref{eq:mu_S formula}).
  \item[(c)] The Beltrami differential of $\zeta_i \circ g_t \circ \zeta_i^{-1}$ is holomorphic as a
  map from $t$ into $L^\infty_{-1,1}$ on $\zeta_i(L_i)$.
  \item[(d)] $g_0$ is the identity on $\riem^P$.
 \end{enumerate}
 Observe that if $g_t$ has these properties, then setting
 \[  F_t = \nu_{\epsilon(t)} \circ g_t : \riem^P \rightarrow \riem_t^P  \]
 \[  f_t = \left. F_t \right|_{\riem}: \riem \rightarrow \riem_t := F_t(\riem)  \]
 we have that
 \[  \mathcal{G} \circ \Pi ([\riem,f_t,\riem_t]) = (\epsilon(t),w_1(t,\cdot),\ldots,w_n(t,\cdot)).  \]
 Since $(\riem,f_0,\riem_0) = (\riem,\mbox{Id},\riem)$ it follows that $(\riem,f_t,\riem_t)$ is a representative of $\lambda(t)$.
 We will show momentarily that $(\riem,f_t,\riem_t)$ has the properties (1), (2), and (3). However we must first
 establish the existence of $g_t$.

 Denote the Ahlfors-Weill extension of $w_i(t,z)$ by $\hat{w}_i(t,z)$.  It is well-known see e.g. \cite{Lehtobook}, that Ahlfors-Weill extension is a holomorphic in $t$ for fixed $z$ so in particular the restriction
 of $t \mapsto \hat{w}_i(t,z)$ to $|z|\leq R_i$ is a
 holomorphic motion.  By the extended lambda lemma \cite{Slodkowski}, denoting by $W$ the region
 enclosed by $\gamma_i$, there is a
 holomorphic motion $H_i:\Delta \times W \rightarrow W$ which equals $\hat{w}_i(t,z)$ on $|z|\leq R$
 and equals the identity on $\gamma_i$.  Setting
 \[  g_t(p) = \left\{ \begin{array}{cc} \zeta_i^{-1} \circ H_i(t, \zeta_i(p)) & p \in \riem \backslash \overline{Y} \\
   p & p \in \overline{Y} \end{array} \right. \]
 we have that $g_t$ satisfies properties (a) through (d).

 Next, we show that $(\riem,f_t,\riem_t)$ has the claimed properties.  The uniform $L^2$ bound can be established
 easily as follows.  Fix $r_i$ such that $0<r_i<R_i$ and let $\Gamma_i = \zeta_i^{-1}(|z|=r_i)$.  Denote by $V_i$ the collar neighbourhood
 bounded by $\partial_i \riem$ and $\Gamma_i$.  The Beltrami differential of $g_t$ satisfies
 $|\mu(g_t)(\zeta_i^{-1}(z))| = |\mu(\zeta_i \circ g_t \circ \zeta_i^{-1})(z)| = |m^{i}_t(1/\bar{z})|$,
 and we have shown that the $L^2$ norm of $m^{i}_t(1/\bar{z})$ is uniformly bounded on $|t|<s'$ and $z \in \disk$.
 Thus $\|\mu(g_t)\|_{2,V_i}$ is uniformly bounded for $|t|<s'$ for all $i$.  Now
 using the fact that $g_t$ is the identity on $\overline{Y}$ and the Schiffer variation $\nu_{\epsilon(t)}$ has zero Beltrami differential
 outside of disks $D_k$ disjoint from $E_i$, we see that
 \[  \mu(f_t) = \mu(\nu_{\epsilon(t)} \circ g_t) = \left\{ \begin{array}{cc} \epsilon(t)d\bar{z}/dz & z \in D_1 \cup \cdots \cup D_d \\
   \mu(g_t) & z \in \riem \backslash Y \\ 0 & \mathrm{otherwise}. \end{array} \right. \]
   (Note that the expression on the Schiffer disks is in terms of the local parameter on each of those disks).
  Thus applying Lemma \ref{le:Lp_separation} with $k=-1$, $l=1$ and $p=2$ and with the collar chart $(\left. \zeta_i \right|_{V_i},V_i)$, and using
 the fact that the $L^\infty$ norm of any Beltrami differential is bounded by one, we have a uniform bound
 on $\|\mu(f_t)\|_{2,\riem}$ for $|t|<s'$. This proves property (1) for $|t|<s'$.

 We now prove the second and third claims.  By assumption, each $w_i(t,z)$ is a holomorphic curve in $\Oqc_{WP}(\disk)$.
 Therefore for any $t_0$ in a sufficiently small neighbourhood $|t|<s''$ of $0$
 there is a holomorphic function $g_{t_0}$ on $\disk$ such that
 \begin{equation} \label{eq:masscentre_temp1}
   \lim_{t \rightarrow t_0} \iint_\disk (1-|z|^2)^2 \left| \frac{\mathcal{S}(w(t,z))- \mathcal{S}(w(t_0,z))}{t-t_0} - g_{t_0}(z) \right|^2
   \,dA =  0.
 \end{equation}
 For the Ahlfors-Weill extension of $w(t,z)$ with dilatation (\ref{eq:mu_S formula}), setting
 \[  \omega^i_{t_0}(1/\bar{z}) = - \frac{(1-|z|^2)^2}{2} \frac{z^2}{\bar{z}^2} \cdot g_{t_0}(z) \]
 we then have that
 \begin{equation} \label{eq:masscentre_temp2}
  \lim_{t \rightarrow t_0} \iint_{\disk} \frac{1}{(1-|z|^2)^2} \left|  \frac{\mu^i_t(1/\bar{z}) - \mu^i_{t_0}(1/\bar{z})}{t-t_0} -
  \omega_{t_0}(1/\bar{z}) \right|^2 =0.
 \end{equation}
 Observe also that since $A_2^2(\disk) \hookrightarrow A_2^\infty(\disk)$ is a bounded inclusion,
 (\ref{eq:masscentre_temp1}) also implies that
 \[  \lim_{t \rightarrow t_0} \left\| (1-|z|^2)^2 \frac{\mathcal{S}(w_i(t,z)) - \mathcal{S}(w_i(t_0,z))}{t-t_0} - g_{t_0}(z) \right\|_\infty = 0  \]
 and hence again by (\ref{eq:mu_S formula})
 \begin{equation} \label{eq:masscentre_temp3}
  \lim_{t \rightarrow t_0} \left\| \frac{\mu^i_t(1/\bar{z}) - \mu^i_{t_0}(1/\bar{z})}{t-t_0} -
  \omega_{t_0}(1/\bar{z}) \right\|_\infty =0.
 \end{equation}

 We first prove that $(\riem,f_t,\riem_t)$ has property (3).
 We need to establish that there is a Beltrami differential $\kappa_{t_0}$ on $\riem$ such that
 \begin{equation} \label{eq:infinty_derivative_on_riem}
  \lim_{t \rightarrow t_0} \left\| \frac{\mu(g_t)-\mu(g_{t_0})}{t-t_0} - \kappa_{t_0} \right\|_\infty =0
 \end{equation}
 holds almost everywhere.  It is enough to show the existence of such a Beltrami differential on
 each portion of the Riemann surface individually.
 Using (\ref{eq:masscentre_temp3}) on $1<|z|<R_i$, and lifting $\omega_{t_0}$ to $\riem$ via $\zeta_i^{-1}$
 establishes the claim on the region bounded by $\zeta_i^{-1}(|z|=R_i)$ and $\partial_i \riem$, since the
 Beltrami differential of $\nu_{\epsilon(t)}$ is zero on this region.
 Property (c) of $g_t$ establishes the claim on the region $L_i$, again using the fact that the
 Beltrami differential of $\nu_{\epsilon(t)}$ is zero there.  Finally, on $Y$ the claim
 follows from the fact that the Beltrami differential of $g_t$ is zero there, and that
 in coordinates the Beltrami
 differential of $\nu_{\epsilon(t)}$ is just $\epsilon(t)d\bar{z}/dz$ on the Schiffer disks $D_k$ and
 zero otherwise; note that $\epsilon(t)$ is a holomorphic function of $t$.
 This establishes property (3) for $|t|<s''$.

 Next we establish property (2).   We will again use Lemma \ref{le:Lp_separation} in the case that
 $k=-1$, $l=1$ and $p=2$.  We again use the collar chart $(\left.\zeta_i \right|_{V_i},V_i)$.  We then
 have the estimate (for some compact $M$ and regions $\mathbb{A}_{r'_i}$ with $r'_i <r_i$)
 \begin{align*}
  \left\| \frac{\mu(f_t)-\mu(f_{t_0})}{t-t_0} - \omega_{t_0} \right\|_{2,\riem} &
    \leq a \left\|\frac{\mu(f_t)-\mu(f_{t_0})}{t-t_0} - \omega_{t_0} \right\|_{\infty,M} \\
      & \ \ + \sum_{i=1}^n b_i\left(  \iint_{\mathbb{A}_{r'_i,1}}
  \lambda_{\mathbb{D}}^{2}(z) \left| \frac{\mu^i_t(1/\bar{z}) - \mu^i_{t_0}(1/\bar{z})}{t-t_0} -
  \omega_{t_0}(1/\bar{z})\right|^2 \right)^{1/2}.  \\
 \end{align*}
 The first term goes to zero by property (3), and the remaining terms go to zero by
 (\ref{eq:masscentre_temp2}).  This establishes property (2) on $|t|<s''.$ Now taking $\delta=\min (s', s'')$  proves the theorem.
 \end{proof}

\end{subsection}
\end{section}
\begin{section}{Tangent space to WP-class Teichm\"uller space and the WP metric}
In this section, we demonstrate the convergence on refined Teichm\"uller space of the
generalized Weil-Petersson metric.
\begin{subsection}{The tangent space to Teichm\"uller space}
 First, we need some results on the function spaces which will serve as models of the
 tangent space.  In all of the following, $\riem$ will be a bordered Riemann surface of
 type $(g,n)$.

 Consider the spaces $L^2_{0,2}(\riem)$ and $L^\infty_{0,2}(\riem)$ of $2$-differentials.
  We have a well-defined mapping from
 these spaces into $L^2_{-1,1}(\riem)$ and $L^\infty_{-1,1}(\riem)$ as follows.
 Let $\psi$ be a two-differential, given in local coordinates $(\zeta,U)$ by
 $\psi_U(z) d\bar{z}^2$.  Assume that the hyperbolic metric on $\riem$ is given by $\rho_U(z)^2 |dz|^2$
 in local coordinates.  It is easily checked that the locally defined functions
 \[  \psi_U(z) \rho_U^{-2}(z)    \]
 transform under change of coordinates as a $(-1,1)$ differential, and hence define a global
 $(-1,1)$ differential.  Denote this map from $2$-differentials to $(-1,1)$-differentials by
 $\mathfrak{B}$.  It's not hard to check that $\mathfrak{B}$ has an inverse (obtained by multiplying by
   $\rho_U(z)^2$ in local coordinates).
   The following property of $\mathfrak{B}$ is an immediate consequence of Definition \ref{de:Lp_definition}
   and the definition of the norm.
 \begin{proposition} \label{pr:B_preserves_p}
   Let $\riem$ be a bordered Riemann surface of type $(g,n)$ and let $\mathfrak{B}$ be defined as above.   For
   any $p \in [1,\infty]$
   \[  \mathfrak{B}(L^p_{0,2}(\riem)) = L^p_{-1,1}(\riem).  \]
  Furthermore, $\| \mathfrak{B}(\alpha) \|_{p} = \| \alpha \|_p$ for any $\alpha \in L^p_{0,2}(\riem)$.
 \end{proposition}
 \begin{remark}  This Proposition generalizes to other $k,\,l$ by dividing by other powers
  of the hyperbolic metric, and also to arbitrary hyperbolic surfaces. However we do not need this here.
 \end{remark}

 We now define the model spaces for the tangent space.
 \begin{definition}
  Let $\mathfrak{B}$ be as above.  Let $\overline{A_2^i(\riem)}$ denote the set of complex
  conjugates of elements of $A_2^i(\riem)$ for $i=2,\infty$.  Let
  \begin{eqnarray*}
   H_{-1,1}(\riem) & = & \mathfrak{B}\left(\overline{A_2^2(\riem)}\right) \\
   \Omega_{-1,1}(\riem) & = & \mathfrak{B}\left(\overline{A_2^\infty(\riem)}\right).
  \end{eqnarray*}
 \end{definition}
 Observe that $\mathfrak{B}$ is a bounded linear isomorphism in both cases, if $H_{-1,1}(\riem)$ and $\Omega_{-1,1}(\riem)$
 are endowed with the norms inherited from $L^2_{-1,1}(\riem)$ and $L^\infty_{-1,1}(\riem)$
 respectively.

 For example, we have that
 \begin{eqnarray*}
  H_{-1,1}(\disk^*) & = & \left\{ (1-|z|^2)^2 \overline{\psi(z)}d\bar{z}/dz \,: \, \iint_{\disk^*} (1-|z|^2)^2 |\psi(z)|^2\,dA < \infty  \right\} \\
  \Omega_{-1,1}(\disk^*) & = & \left\{ (1-|z|^2)^2 \overline{\psi(z)} d\bar{z}/dz \,: \, \sup_{z \in \disk^*} (1-|z|^2)^2 |\psi(z)| < \infty   \right\}.
 \end{eqnarray*}

 The space $\Omega_{-1,1}(\riem)$ is well-known to be complementary to the so-called infinitesimally trivial
 Beltrami differentials.  It was shown by Takhtajan and Teo \cite{Takhtajan_Teo_Memoirs} that $H_{-1,1}(\disk^*)$
 is the tangent space to the WP-class universal Teichm\"uller space (although they do not use that term), and $\twp(\mathbb{D}^*)$ can be modelled by $H_{-1,1}(\disk^*)$.  Furthermore, the Weil-Petersson metric converges on
 this tangent space.  We will show that this is true for $H_{-1,1}(\riem)$, if one uses the Hilbert manifold
 structure which the authors defined in \cite{RSS_Hilbert}.
 \begin{remark} \label{re:no_lift_two}
  Lifting to the cover $\mathbb{D}^*$ of $\riem$ and attacking this problem using differentials
  which are invariant under the group of deck transformations does not appear to provide any advantage.   This is
  because
  the relevant lifted differentials are $L^2$ on the fundamental domain but not on $\disk^*$.
  That is, $H_{-1,1}(\riem)$ is not the set of invariant differentials in $H_{-1,1}(\disk^*)$.  This
  is an unavoidable consequence of replacing the $L^\infty$ norm with an $L^2$ norm.
  See also Remark \ref{re:no_lift_one}.
 \end{remark}

 Before proceeding, we must establish some analytic results.  The main result which we require is the following.
 \begin{theorem} \label{th:H_Omega_inclusion}
  If $\riem$ is a bordered Riemann surface of type $(g,n)$, then $H_{-1,1}(\riem) \subset \Omega_{-1,1}(\riem)$.
  Furthermore, the inclusion map is bounded.
 \end{theorem}
 First, we require two lemmas.
 \begin{lemma} \label{le:Wulfs_lemma} Let $f:\mathbb{A}_r \rightarrow \mathbb{C}$ be a holomorphic function on $\mathbb{A}_{r}.$
 Then for any $t\in(1,r)$
  \begin{equation} \label{eq:Wulf_estimate}
    \sup_{z \in \mathbb{A}_t} (1-|z|^2)^2|f(z)| \leq C(r,t) \left(  \iint_{\mathbb{A}_r} (1-|z|^2)^2 |f(z)|^2 \,dA
     \right)^{1/2}
  \end{equation}
  with
  \[  C(r,t)= \frac{4}{\sqrt{2\pi}} \frac{\sqrt{r^2 +t^2}}{r^2-t^2}\frac{(1-t^2)^2 }{r^2 +3} +4\frac{\sqrt{3}}{\sqrt{\pi}} t.  \]
  \end{lemma}

 \begin{proof} To prove this lemma,
  let $f(z)=\sum_{n=-\infty}^{\infty} a_{n} z^{n}$ be the Laurent series of $f$ in $\mathbb{A}_r$. An elementary calculation reveals that
 \begin{equation}\label{norm of f in A2}
  \iint_{\mathbb{A}_{r}} |f(z)|^{2}(1-|z|^2)^2 dA= 2\pi\sum_{n=-\infty}^{\infty} |a_n|^{2} I_{n}(r),
 \end{equation}
 where $I_{n}(r):=\frac{1}{2}\int_{1}^{r^{2}}\rho^{n}(1-\rho)^{2}\, d\rho.$ Now the Cauchy-Schwarz inequality yields
\begin{align} \label{ norm of f in Ainfty}
|f(z)| &\leq \left(\sum_{n=-\infty}^{\infty}|a_{n}|^{2} I_{n}(r) \right)^{1/2}\left(\sum_{n=-\infty}^{\infty} \frac{|z|^{2n}}{ I_{n}(r)} \right)^{1/2} \nonumber \\ &=\frac{1}{\sqrt{2\pi}}\left(  \iint_{\mathbb{A}_r} (1-|z|^2)^2 |f(z)|^2 \,dA
   \right)^{1/2} \left(\sum_{n=-\infty}^{\infty} \frac{|z|^{2n}}{ I_{n}(r)} \right)^{\frac{1}{2}}.
\end{align}

Now let us estimate the quantity $\sum_{n=-\infty}^{\infty} \frac{|z|^{2n}}{ I_{n}(r)}$. To this end we split the sum as follows
\begin{equation}
  \sum_{n=0}^{\infty} \frac{|z|^{2n}}{ I_{n}(r)}+\sum_{n=1}^{\infty} \frac{|z|^{-2n}}{I_{-n}(r)}:= \textbf{I}+\textbf{J}.
\end{equation}

We observe that since on $\mathbb{A}_{r}$ we have $1<|z|<r$, then $0<\frac{|z|^{2}-r^2}{1-r^2}<1$ on $\mathbb{A}_{r}$. Bearing this fact in mind we proceed with the estimates of the above terms.

To estimate $\textbf{I}$, take an $s$ with $0<s<\frac{|z|^{2}-r^2}{1-r^2}$ and set $r_{s}=s+(1-s)r^2$, for that choice of $s$. This yields that

\begin{equation}
I_{n}(r)\geq \frac{1}{2}\int_{r_{s}}^{r^2} \rho^{n} (\rho-1)^{2}\, d\rho\geq \frac{r_{s}^{n}}{2}(r_s-1)^{2} \int_{r_{s}}^{r^2} d\rho=\frac{r_{s}^{n}}{2}(r_s-1)^{2} (r^2-r_{s})= \frac{r_{s}^{n}}{2} s(1-s)^{2}(r^2 -1)^{3}.
\end{equation}

Therefore since $0<s<\frac{|z|^{2}-r^2}{1-r^2}$ implies that $\frac{|z|^2}{r_s}<1$, we have
\begin{equation}\label{I estim}
 \textbf{I} \leq \frac{2}{s(1-s)^2}\frac{1}{(r^2 -1)^3} \sum_{n=0}^{\infty} \left(\frac{|z|^2}{r_{s}}\right)^{n} = \frac{2}{s(1-s)^2}\frac{1}{(r^2 -1)^3} \frac{r_{s}}{r_{s}- |z|^2 }.
\end{equation}

Now we turn to the estimate for $\textbf{J}$, to this end take an $s'$ with $\frac{|z|^{2}-r^2}{1-r^2}<s'<1$ and set $r_{s'}:=s'+(1-s')r^2$ for that choice of $s'$. This yields that

\begin{equation}
I_{-n}(r)\geq \frac{1}{2}\int_{1}^{r_{s'}} \rho^{-n} (\rho-1)^{2}\, d\rho\geq \frac{r_{s'}^{-n}}{2} \int_{1}^{r_{s'}}(\rho-1)^{2}\, d\rho= \frac{r_{s'}^{-n}}{6} (1-s')^{3}(r^2 -1)^{3}.
\end{equation}
Hence since $\frac{|z|^{2}-r^2}{1-r^2}<s'<1$ implies that $|z|^{-2}r_{s'} <1$
\begin{equation}\label{J estim}
 \textbf{J} \leq \frac{6}{(1-s')^3}\frac{1}{(r^2 -1)^3} \sum_{n=0}^{\infty} (|z|^{-2} r_{s'})^{n} = \frac{6}{(1-s')^3}\frac{1}{(r^2 -1)^3} \frac{|z|^2 }{|z|^2 - r_{s'}}.
\end{equation}
Now \eqref{I estim} and \eqref{J estim} yield that

\begin{multline}\label{part of the Linfty norm}
  (1-|z|^2)^{2}\left(\sum_{n=-\infty}^{\infty} \frac{|z|^{2n}}{ I_{n}(r)}\right)^{\frac{1}{2}} \leq \frac{\sqrt{2}}{\sqrt{s}(1-s)}\frac{ (1-|z|^2)^{2}}{(r^2 -1)^{3/2}} \frac{\sqrt{r_{s}}}{\sqrt{r_{s}- |z|^2} }+\\ \frac{\sqrt{6}}{(1-s')^{3/2}}\frac{ (1-|z|^2)^{2}}{(r^2 -1)^{3/2}} \frac{|z|}{\sqrt{|z|^2 - r_{s'}}}:=R_1 +R_2.
\end{multline}
 Now since the inequality \eqref{part of the Linfty norm} is valid for all $s\in (0,\frac{|z|^{2}-r^2}{1-r^2})$ and $s'\in (\frac{|z|^{2}-r^2}{1-r^2}, 1)$, we take $s= \frac{|z|^{2}-r^2}{2(1-r^2)}$, $s'=\frac{1}{2}+ \frac{|z|^2-r^2}{2(1-r^2)}$. With these choices of $s$ and $s'$ we have, $1-s= \frac{|z|^{2}+r^2 -2}{2(r^2-1)}$, $r_s =\frac{r^2 +|z|^{2}}{2}$, $r_s - |z|^2 =\frac{r^2 -|z|^{2}}{2}$ and $1-s'=\frac{1}{2}(\frac{|z|^2 -1}{r^2 -1}),$ $r_{s'} =\frac{1}{2} (1+|z|^2)$, $|z|^2-r_{s'} = \frac{1}{2}(|z|^2 -1).$ Plugging in these values into $R_1$ and $R_2$ yields that $R_1=4\frac{\sqrt{r^2 +|z|^2}}{r^2-|z|^2}\frac{(1-|z|^2)^2}{|z|^2+ r^2 +2}$ and $R_2 = 4\sqrt {6}|z|$. This and \eqref{ norm of f in Ainfty} yield for $z\in\mathbb{A}_r$ the pointwise estimate

\begin{equation}\label{pointwise estim}
(1-|z|^2)^2 |f(z)| \leq \left (\frac{4}{\sqrt{2\pi}} \frac{\sqrt{r^2 +|z|^2}}{r^2-|z|^2}\frac{(1-|z|^2)^2 }{|z|^2+ r^2 +2} +4\frac{\sqrt{3}}{\sqrt{\pi}}  |z|\right ) \left(  \iint_{\mathbb{A}_r} (1-|z|^2)^2 |f(z)|^2 \,dA
     \right)^{1/2} .
\end{equation}

Now observing that for $1<|z|<t$ one has $$\frac{4}{\sqrt{2\pi}} \frac{\sqrt{r^2 +|z|^2}}{r^2-|z|^2}\frac{(1-|z|^2)^2 }{|z|^2+ r^2 +2} +4\frac{\sqrt{3}}{\sqrt{\pi}}  |z|\leq \frac{4}{\sqrt{2\pi}} \frac{\sqrt{r^2 +t^2}}{r^2-t^2}\frac{(1-t^2)^2 }{r^2 +3} +4\frac{\sqrt{3}}{\sqrt{\pi}} t,$$ taking the supremum in \eqref{pointwise estim} over all $z\in\mathbb{A}_t$ yields the desired estimate.
\end{proof}
\begin{lemma} \label{le:Bergman_estimate_interior}
 Let $\riem$ be a bordered Riemann surface of type $(g,n)$ and let $M$ be compactly contained in $\riem$.
 There is a constant $D_M$ depending only on $M$ such that
 for any $\alpha \in H_{-1,1}(\riem)$,
 \[  \| \alpha \|_{\infty,M,\riem} \leq D_M \| \alpha \|_{2,\riem}.  \]
\end{lemma}
\begin{proof}
 Since $M$ is compact in $\riem$, there is a compact subset $N$ of $\riem$ such that $M$ is a subset of the interior
 of $N$.  There is a finite collection of open sets $W_k$ and $V_k$, $k=1,\ldots,m$
 such that
 \begin{enumerate}
  \item $\overline{W_k} \subseteq V_k \subseteq N$ for $k=1,\ldots,m$ where $\overline{W_k}$
   denotes the closure of $W_k$,
  \item There are coordinate charts $\eta_k:V_k \rightarrow G_k$, where $G_k$ are bounded, open
   connected subsets of $\mathbb{C}$, and
  \item $M \subseteq \cup_{k=1}^m W_k$.
 \end{enumerate}

 Since $N$ is a compact subset of $\riem$, and there are only finitely many charts $(\eta_k,V_k)$,
 there is a constant $C>0$ such that if $\rho_{V_k}$ denotes the hyperbolic metric in local
 coordinates then
 \begin{equation} \label{eq:interior_estimate_temp}
  \frac{1}{C} \leq \rho_{V_k}(z) \leq C
 \end{equation}
 for each $k=1,\ldots,m$.

 Now let $\alpha \in H_{-1,1}(\riem)$.  There is a $\psi \in A_2^2(\riem)$ such that in $\eta_k$
 coordinates,
 $\alpha$ has the form $\rho_{V_k}(z)^{-2} \overline{\psi_{V_k}(z)}d\bar{z}/dz$
 where $\psi_{V_k}(z) dz^2$ is the expression for $\psi$ in $\eta_k$ coordinates.  For all $z \in W_k$,
 we have by an elementary estimate that there is a constant $E_k$, which is independent of $\psi$
 (depending only on $\eta_k(W_k)$ and $\eta_k(V_k)$) such that
 \[  |\psi_{V_k}(z)| \leq E_k \left( \iint_{\eta_k(V_k)} |\psi_{V_k}(z)|^2 \,dA \right)^{1/2}    \]
 where $dA$ denotes the area element $d\bar{z} \wedge dz /2i$.
 This can be obtained by applying the mean value property of holomorphic functions.
 Thus, applying \eqref{eq:interior_estimate_temp} twice, we see that
 \begin{align*}
  \rho_{V_k}^{-2}(z) |\psi_{V_k}(z)| & \leq C^2 | \psi_{V_k}(z)|  \\
  & \leq C^2 E_k \left( \iint_{\eta_k(V_k)} |\psi_{V_k}(z)|^2\,dA \right)^{1/2} \\
  & \leq C^3 E_k \left( \iint_{\eta_k(V_k)} \rho_{V_k}(z)^{-2} |\psi_{V_k}(z) |^2 \,dA \right)^{1/2} \\
  & \leq C^3 E_k \| \alpha \|_{2,\riem}.
 \end{align*}
 Since $W_k$ cover $M$, taking
 \[  D_M = C^3 \max \{ E_1,\ldots,E_k \} \]
 the claim is proven.
\end{proof}

\begin{proof}(of Theorem \ref{th:H_Omega_inclusion}).
 Choose collar charts $(\zeta_i,U_i)$ and annuli $\mathbb{A}_{r_i}$ satisfying the conclusion of
 Lemma \ref{le:collar_hyperbolic_singularity}
 with constants $K_i$.  Since there are only finitely many boundary curves we may assume that $K_i=K$
 for some $K$ for all $i$.  Choose $s_i$ such that $1< s_i <r_i $ for each $i$, and let $M$ be
 the subset of $\riem$ given by
 \[   M = \riem \backslash \cup_{i=1}^n \{ \zeta_i^{-1}(\mathbb{A}_{s_i}) \}.  \]
 Clearly $M$ is compactly contained in $\riem$.

 Let $\alpha \in H_{-1,1}(\riem)$.
 Applying Lemma \ref{le:Bergman_estimate_interior} we have that
 \[  \| \alpha \|_{\infty,M,\riem} \leq D_M \| \alpha \|_{2,\riem} \]
 where $D$ depends only on $M$.  On the other hand, given $\alpha$ by definition there is a
 $\psi \in A^2_2(\riem)$ such that $\alpha$ locally has the expression $\rho_U(z)^{-2} \overline{\psi_U(z)}$,
 where $\psi_U$ is the local expression for $\psi$.
 In particular $\alpha$ has the form  $\rho_{U_i}(z)^{-2} \overline{\psi_{U_i}(z)}$ in $\zeta_i$ coordinates.

 For each $i=1,\ldots,n$, choose $t_i$ such that $1<t_i<s_i$.
 By Lemmas \ref{le:Wulfs_lemma} and \ref{le:collar_hyperbolic_singularity},
 we have the estimate
 \begin{align*}
  \sup_{z \in \mathbb{A}_{r_i}} \rho_{U_i}(z)^{-2} | \psi_{U_i}(z) | & \leq K^2 \sup_{z \in \mathbb{A}_{r_i}} \lambda_{\disk}(z)^{-2} |\psi_{U_i}(z)|
  \\ & \leq K^2 C(r_i,t_i) \left(\iint_{\mathbb{A}_{r_i}} \lambda_{\disk}(z)^{-2} |\psi_{U_i}(z)|^2 \,dA \right)^{1/2}
  \\ & \leq K^3 C(r_i,t_i) \left( \iint_{\mathbb{A}_{r_i}} \rho_{U_i}(z)^{-2} |\psi_{U_i}(z)|^2 \,dA \right)^{1/2} \\
  & \leq K^3 C(r_i,t_i) \| \alpha \|_{2,\riem}.
 \end{align*}
 If we let $C = \max \{ D_M, K^3 C(r_1,t_1),\ldots,K^3 C(r_n,t_n) \}$ we thus have proven that
 \[  \| \alpha \|_{\infty,\riem} \leq C \| \alpha \|_{2,\riem}. \]
\end{proof}
Next, we observe an immediate but important corollary.
\begin{corollary} \label{co:intersection_good}
 Let $\riem$ be a bordered Riemann surface of type $(g,n)$.
 \[  \Omega_{-1,1}(\riem) \cap \tbd(\riem) = H_{-1,1}(\riem).  \]
\end{corollary}

The remainder of this section is dedicated to showing that $H_{-1,1}(\riem)$ is
complementary to the kernel of the restriction of the Bers embedding to $\tbd(\riem)$.
In the next section, we will show that it
is a model for the tangent space to WP-class Teichm\"uller space.

  First, we will recapitulate some of the known facts regarding the kernel of the Bers
  embedding in the case of the standard Teichm\"uller space.
  One has the standard decomposition
  \[  L_{-1,1}^\infty(\mathbb{D}^*) = \mathcal{N}(\disk^*) \oplus \Omega^{-1,1}(\disk^*)  \]
  where
   \[  \mathcal{N}(\disk^*)=\left\{ \mu \in L^\infty_{-1,1}(\disk^*) \,:\,
    \iint_{\disk^*} \mu \phi =0  \ \ \forall \phi \in A^1_2(\disk^*) \right\}  \]
  is the linear space of ``infinitesimally trivial'' Beltrami differentials, which are
  by classical results precisely the kernel of the derivative at the identity of the
  Bers embedding of the universal Teichm\"uller space \cite[Chpt 3]{Nagbook}, \cite[V.7]{Lehtobook}.
  We now define similar spaces on the bordered Riemann surface $\riem$ of type $(g,n)$.
  \[  \mathcal{N}(\riem)=\left\{ \mu \in L^\infty_{-1,1}(\riem) \,:\,
    \iint_{\riem} \mu \phi =0  \ \ \forall \phi \in A^1_2(\riem) \right\}.  \]

  The following theorem is standard, although often phrased in its equivalent form
  using Fuchsian groups.
  \begin{theorem} \label{th:decomp_on_riemB}
   Let $\riem$ be a bordered Riemann surface of type $(g,n)$.
   Then
   \[  L_{-1,1}^\infty(\riem) = \mathcal{N}(\riem) \oplus \Omega_{-1,1}(\riem).  \]
   Furthermore $\mathcal{N}(\riem)$ is the kernel of the Bers embedding.
  \end{theorem}

  We sketch the proof in order to establish the
  notation and concepts.  Full details can be found in the references.
  We have (up to biholomorphism) that $\riem = \disk^*/G$ for some Fuchsian group $G$.
  Define
  \[  L^\infty_{-1,1}(\mathbb{D}^*,G) = \left\{ \mu \in L_{-1,1}^\infty(\disk^*) \,:\,  \mu \circ g \frac{\overline{g'}}{g'} = \mu
   \ \ \forall g \in G \right\}.  \]
  Also define
  \[  \mathcal{N}(\disk^*,G) = \mathcal{N}(\disk^*) \cap L^\infty_{-1,1}(\mathbb{D}^*,G)  \]
  and
  \[  \Omega_{-1,1}(\disk^*,G) = \Omega_{-1,1}(\disk^*) \cap L^\infty_{-1,1}(\mathbb{D}^*,G).  \]
  It is immediate that
  \begin{equation} \label{eq:decomp_with_G}
   L^\infty_{-1,1}(\mathbb{D}^*,G) = \mathcal{N}(\disk^*,G) \oplus
     \Omega_{-1,1}(\disk^*,G).
  \end{equation}

  Clearly we may identify $L^\infty_{-1,1}(\mathbb{D}^*,G)$ with $L^\infty_{-1,1}(\riem)$ and $\Omega_{-1,1}(\disk^*,G)$
  with $\Omega_{-1,1}(\riem)$.  We must show that $\mathcal{N}(\disk^*,G)$ can be identified with $\mathcal{N}(\riem)$,
  and that $\mathcal{N}(\riem)$ is the kernel of the derivative of the Bers embedding at the point $[\riem,\text{Id},\riem]$.
  This latter fact is well known: let $F$ be a fixed fundamental domain of the group $G$ and temporarily let
  \[  A^1_2(F) = \left\{ \phi(z) \in L^2(\disk^*) \,: \, \phi \ \mathrm{holo}, \ \phi(g(z)) g'(z)^2 dz^2 = \phi(z) \ \ \forall g \in G, \ \mbox{and}
     \iint_{F} |\phi(z)|\,dA < \infty \right\} \]
  and temporarily let
  \[  N(G) = \left\{ \mu \in L_{-1,1}^\infty(\disk^*,G) \,:\, \iint_F \mu \phi =0 \
 \forall \phi \in A^1_2(F) \right\}.  \]  It is clear that $N(G)$
 can be identified with $\mathcal{N}(\riem)$.  By \cite[Chapter V, Theorem 7.2]{Lehtobook}, $N(G)$ is the
 kernel of the derivative of the Bers embedding at the base point.
 It remains to show that $\mathcal{N}(\disk^*,G)$ can be identified with $\mathcal{N}(\riem)$.

 To do this we show that $\mathcal{N}(\disk^*,G) = N(G)$.  Let $\Theta:A^1_2(\disk^*) \rightarrow A_2^1(F)$ be the Poincar\'e
 projection operator \cite[V.7.3]{Lehtobook}.
 Let $\mu \in \mathcal{N}(\disk^*) \cap L_{-1,1}^\infty(\disk^*,G)$.  Let $\phi \in A^1_2(F)$.
 Since $\Theta$ is surjective \cite[Theorem V.7.1]{Lehtobook} there is a $\psi \in A_2^1(\disk^*)$ such
 that $\Theta(\psi)=\phi$.  By \cite[Chapter V, (7.3)]{Lehtobook}
 \[  \iint_F \mu \phi = \iint_{\disk^*} \mu \psi =0, \]
 so $\mu \in N(G)$.

 Conversely assume that $\mu \in N(G)$.  Let $\phi \in A_2^1(\disk^*)$.  By   \cite[Equation (7.3) V.7.3]{Lehtobook}
 \[  \iint_{\disk^*} \mu \phi = \iint_F \mu (\Theta \phi) =0. \]
 So $\mu \in \mathcal{N}(\disk^*) \cap L_{-1,1}^\infty(\disk^*,G)$.
 Thus $\mathcal{N}(\disk^*,G)=N(G)$.

 We conclude that
 \begin{equation}  \label{eq:decomp_nonrefined}
   L^\infty_{-1,1}(\riem) = \mathcal{N}(\riem) \oplus
     \Omega_{-1,1}(\riem).
 \end{equation}

 It now follows that $H_{-1,1}(\riem)$ is complementary to the infinitesimally trivial differentials.  Define
 \[  \mathcal{N}_r(\riem)= \mathcal{N}(\riem) \cap \mbox{TBD}(\riem) = \mathcal{N}(\riem) \cap L^2_{-1,1}(\riem).  \]
 Thanks to Corollary \ref{co:intersection_good} of Theorem \ref{th:H_Omega_inclusion}, we have the following theorem.
 \begin{theorem} \label{th:infinitesimal_trivial_decomposition} Let $\riem$ be a bordered Riemann surface of type $(g,n)$.
  \[ \tbd(\riem) = \mathcal{N}_r(\riem) \oplus H_{-1,1}(\riem).  \]
  Furthermore, $\mathcal{N}_r(\riem)$ is the intersection of the kernel of the Bers
  embedding at the base point with $L^2_{-1,1}(\riem)$; that is, it is the kernel
  of the restriction of the Bers embedding to $\bd(\riem)$.
 \end{theorem}
 \begin{proof}  The final statement follows from the above discussion.  The decomposition follows
  from Theorem \ref{th:decomp_on_riemB} by intersecting with $L^2_{-1,1}(\riem)$ and applying Corollary \ref{co:intersection_good}.
 \end{proof}
    We also require the following classical result.
    Let $G$ be a Fuchsian group acting on $\mathbb{D}$, and let $F$ be a fundamental domain for $G$.
    (We choose the cover $\disk$ rather than $\disk^*$ for the next few paragraphs
    in order to be consistent with the references and avoid minor convergence issues).     For
    $(-1,1)$ differentials $\nu$-invariant under $G$, define the integral map
    \begin{equation} \label{eq:Kdefinition}
     \mathbf{K}(\nu)(z)=  \frac{3}{\pi} (1-|z|^2)^2 \iint_{\disk} \frac{1}{(1-\bar{\zeta} z)^4}
     \overline{\nu(\zeta)} \,dA_\zeta \ .
    \end{equation}
   We have not yet addressed convergence.  We claim that
   \begin{equation}
    \mathbf{K}: L^\infty_{-1,1}(\mathbb{D},G) \longrightarrow \Omega_{-1,1}(F)
   \end{equation}
   and
   \begin{equation}
    \mathbf{K}: L^2_{-1,1}(\mathbb{D},G) \longrightarrow H_{-1,1}(F)
   \end{equation}
   are bounded, where $L^2_{-1,1}(\mathbb{D},G)$ is the space of $G$-invariant
   $(-1,1)$ differentials such that
   \[  \| \nu \|_{2,F} = \iint_{F} \frac{ |\nu(z)|^2}{(1-|z|^2)^2} \,dA <\infty.   \]
   This can be identified with $L^2_{-1,1}(\riem)$ if $\riem=\disk/G$.
   This follows from \cite[Lemma 3.4.9]{Lehner} and Proposition \ref{pr:B_preserves_p}.
   Furthermore, the kernel of $\left. \mathbf{K} \right|_{L^\infty_{-1,1}(\mathbb{D},G)}$
   is just the infinitesimally trivial differentials $\mathcal{N}(F)$.

   It is clear that these results can be written on the Riemann surface $\riem$
   rather than the fundamental domain.
   Restating the above results on $\riem$, and applying
   Theorem \ref{th:infinitesimal_trivial_decomposition}, we have the following
   theorem.
   \begin{theorem} \label{th:bounded_projection} Let $\riem$ be a bordered
  Riemann surface of type $(g,n)$. There is a bounded projection
    $P: L^\infty_{-1,1}(\riem) \rightarrow \Omega_{-1,1}(\riem)$
    such that the restriction
    \[  \left. P \right|_{\bd(\riem)}: \bd(\riem) \rightarrow H_{-1,1}(\riem)  \]
    is bounded with respect to the $L^2_{-1,1}(\riem)$ norm.
    The kernel of the restriction of $P$ to $\tbd(\riem)$ is
    $\mathcal{N}_r(\riem)$.
   \end{theorem}
   \begin{remark}  In the last statement, we make use of the fact that
    the derivative of $P$ is $P$ itself.  Note that $P$ is linear on both
    spaces $L^2_{-1,1}(\riem)$ and $L^\infty_{-1,1}(\riem)$.
   \end{remark}

 When combined with Theorem \ref{thm:local_bounded_one_param_Upp}, we get the following
 crucial consequence.
 \begin{theorem} \label{th:tangent_in_H11} Let $\riem$ be a bordered
  Riemann surface of type $(g,n)$ such that $2g-2+n>0$.
 Assume that $\mathbf{v}$ is a tangent vector to $\twp(\riem)$ at $[\riem,\text{Id},\riem]$.
  There is a holomorphic curve $t \mapsto [\riem,f_t,\riem_t]$, $|t|<\delta$, in $\twp(\riem)$ through
  $[\riem,\text{Id},\riem]$ at $t=0$
  such that the Beltrami differential $\mu_t$ of $f_t$ is in $H_{-1,1}(\riem)$ for all $|t|<\delta$,
  $\mu_t$ is holomorphic in $t$, and such that
  the tangent vector to this curve at $[\riem,\text{Id},\riem]$ is $\mathbf{v}$.
 \end{theorem}
 \begin{proof}
  Let $\mathbf{v}$ be a tangent vector to $\twp(\riem)$ at $[\riem,\text{Id},\riem]$.
  Let $[\riem,f_t,\riem_t]$ be a holomorphic curve through $[\riem,\text{Id},\riem]$ at
  $t=0$.  By Theorem \ref{thm:local_bounded_one_param_Upp} we can assume that the Beltrami
  differential of $f_t$ is in $\bd(\riem)$ for $|t|<\delta$ for some $\delta>0$.  By Theorem \ref{th:bounded_projection}
  if we set $\mu_t = P(\mu(f_t))$ the resulting Beltrami differential is in $H_{-1,1}(\riem)$
  for all $|t|<\delta$.
  Furthermore solving the Beltrami equation to obtain $[\riem,g_t,\riem_t]$,
  the tangent vector to $[\riem,g_t,\riem_t]$ at $t=0$ must be the same as $[\riem,f_t,\riem_t]$
  by Theorem \ref{th:infinitesimal_trivial_decomposition},
  since infinitesimally trivial differentials are in the kernel of the Bers embedding.
 \end{proof}
\end{subsection}
\begin{subsection}{$H_{-1,1}(\riem)$ model of Weil-Petersson class Teichm\"uller space}
 In the previous section we showed that the tangent vector at the identity of every differentiable curve
 through the base element of the WP-class Teichm\"uller space
 is in $H_{-1,1}(\riem)$.  In this section, we show
 that the WP-class Teichm\"uller space is locally biholomorphic to $H_{-1,1}(\riem)$. This is the main result of the paper and it allows us to define a convergent WP-metric in Section \ref{explicit WP}.
 This also gives an alternate description of the complex structure.
 
 We proceed as follows. First, in Theorem \ref{thm:holo_H11_coords} below, we show that the restriction to $H_{-1,1}(\riem)$ of the map $\Phi$ taking Beltrami differentials to the solution of the Beltrami equation is holomorphic into $T_{\mathrm{WP}}(\riem)$ on some open neighborhood of $0$. This result uses the theory of marked holomorphic families along with the characterization of the hyperbolic $L^2$ norm in terms of collar charts obtained in Section \ref{se:local_charact}. Once this is established, we apply the preparation Theorem \ref{thm:local_bounded_one_param_Upp} together with the inverse function theorem to establish that in fact $\Phi$ is a biholomorphism on some open ball. In \cite{RSS_Hilbert} we showed that change of base point is a biholomorphism. Using this fact allows us to show in Theorem \ref{th:last bloody theorem} that any 
 point has a neighbourhood biholomorphic to a ball in $H_{-1,1}(\riem)$.

  Define the map
 \begin{align*}
  \check{\Phi}:\Omega_{-1,1}(\riem) & \rightarrow T(\riem) \\
  \mu & \mapsto [\riem,f_\mu,\riem_1].
 \end{align*}
where $f_\mu:\riem \rightarrow \riem_1$ is a solution
 to the Beltrami equation with differential $\mu$.
 Let
 \[  \Phi:H_{-1,1}(\riem) \rightarrow \twp(\riem)  \]
 be the restriction of $\check{\Phi}$ to $H_{-1,1}(\riem)$.  Note that since $H_{-1,1}(\riem) \subseteq L^2_{-1,1}(\riem)$,
 by Theorem \ref{th:Teich_WP_BD_model} $\Phi$ maps into $\twp(\riem)$.  We will keep the distinction
 between $\Phi$ and $\check{\Phi}$, even though $\Phi$ is the restriction of $\check{\Phi}$, in order to
 indicate the change in norm on the domain.

\begin{theorem} \label{thm:holo_H11_coords} Let $\riem$ be a bordered
  Riemann surface of type $(g,n)$ such that $2g-2+n>0$. Then there is an open neighbourhood $B$ of $0 \in H_{-1,1}(\riem)$, such that the
 map $\Phi$
 is holomorphic on $B$.
\end{theorem}
\begin{proof}
 Fix a $\tau \in \qswp(\riem)$ (which exists by Theorem \ref{th:WPclass_param_exists})
 and sew caps on $\riem$ via $\tau$ to obtain a punctured Riemann surface $\riem^P$.
 We assume moreover that $\tau$ is an analytic parametrization, so that $\tilde{\tau}_i$
 has an analytic extension to an open neighbourhood of $\overline{\disk}$ for each $i=1,\ldots,n$.
 Since the complex structure is independent of $\tau$, there is no loss of generality in this assumption.

 Let $\hat{B}$ be the open unit ball in $\Omega_{-1,1}(\riem)$ centred at $0$.  Since inclusion $\iota: H_{-1,1}(\riem) \rightarrow
 \Omega_{-1,1}(\riem)$ is holomorphic by Theorem \ref{th:H_Omega_inclusion}, we see that $B=\iota^{-1}(\hat{B})$ is open in $H_{-1,1}(\riem)$
 and contains $0$.
 Given a $\mu \in B$, we let $\hat{\mu}$ be the Beltrami differential obtained from $\mu$ by setting it to be
 zero on the caps.
 Define a map into the Teichm\"uller space
 $T(\riem^P)$  by
 \begin{align*}
  \Xi: B & \longrightarrow T(\riem^P) \\
  \mu & \longmapsto [\riem^P,f_{\hat{\mu}},\riem_{\hat{\mu}}^P]
 \end{align*}
 where $\riem_{\hat{\mu}}^P$ and  $f_{\hat{\mu}}$ are defined as in expressions (\ref{eq:canonicalrep}) and (\ref{eq:f_mu}).
 In particular, $f_{\hat{\mu}}:\riem^P \rightarrow \riem_{\hat{\mu}}^P$ is a solution to the Beltrami equation on $\riem^P$ with
 dilatation $\hat{\mu}$.  We claim that $\Xi$ is holomorphic.  This is because it
 can be written as the composition of four holomorphic maps.  That is,
 $\Xi = \Psi \circ \mathrm{ext} \circ i \circ \iota$
 where (1) inclusion $\iota:H_{-1,1}(\riem) \rightarrow \Omega_{-1,1}(\riem)$ is holomorphic
 by Theorem \ref{th:H_Omega_inclusion}; (2) the inclusion $i:\Omega_{-1,1}(\riem) \rightarrow L^{\infty}_{-1,1}(\riem)$
 is obviously holomorphic; (3) $\mathrm{ext}:L^\infty_{-1,1}(\riem) \mapsto L^\infty_{-1,1}(\riem^P)$ is holomorphic since
 by direct computation it is G\^ateaux holomorphic and locally (in fact, globally) bounded;   
 finally (4) the solution to the Beltrami equation $\Psi:L^\infty_{-1,1}(\riem^P) \rightarrow
 T(\riem^P)$ is holomorphic (see \cite{Nagbook}). Observe that
 $\Xi= \mathcal{F} \circ \Pi \circ \Phi$; however note that this factorization was not necessary in
 the forgoing proof of holomorphicity of $\Xi$ .
 
A word on the proof may be helpful.  In order to write $\Phi$ in
coordinates $\mathcal{G} \circ \Phi$, we must write points in
the image of $\Pi \circ \Phi$ in terms of a coordinate system
$F(U,S,\Omega)$.   However, a given point in the image of
$\Pi \circ \Phi$ will not be of the form $(\riem^P,\nu_\epsilon \circ f,\riem_\epsilon^P,\psi)$; thus we need to compose by some biholomorphism $\sigma_\epsilon$
of $\riem_\epsilon^P$ and invoke the Teichm\"uller equivalence under homotopy to reach this form.   This explains the presence of
some extra compositions.  Furthermore, in order to do this, and make
rigorous statements regarding holomorphicity, we must the the theory
of marked holomorphic families.  The reader should bear these two
points in mind in what follows. 

Fix an $n$-chart $(\zeta,E)$ on $\riem^P$ such that $\overline{\tilde{\tau}_i(\disk)} \subseteq \zeta_i(E_i)$
 for each $i=1,\ldots,n$.
 Choose an open set $K_i \subset \zeta_i(E_i)$ containing $\overline{\tilde{\tau}_i(\disk)}$ for
 $i=1,\ldots,n$ such that $\overline{K_i}$ is compactly
 contained in $\zeta_i(E_i)$.  Let $U_i \subset \Oqcwp$ be open sets chosen so that $\overline{\phi_i(\disk)}
 \subseteq K_i$ for all $\phi_i$ in $U_i$, $i=1,\ldots,n$. This is possible by \cite[Theorem 3.4]{RSS_Hilbert}.
  Let $\mathfrak{S}:\Omega \rightarrow T(\riem^P)$ be a Schiffer variation compatible with the $n$-chart and let
  $F(U,S,\Omega)$ be the corresponding open set in $\ttwp(\riem^P)$.
 Let $\pi:S(\Omega,D) \rightarrow \Omega$ be the marked Schiffer family corresponding to $\mathfrak{S}$, with strong global trivialization
 \begin{align*}
  \theta:\Omega \times \riem^P & \rightarrow S(\Omega,D) \\
  (\epsilon,q) & \mapsto (\epsilon,\nu_\epsilon(q)),
 \end{align*}
as defined in (\ref{eq:strong_trivialization}).

 By Theorem \ref{th:Schiffer_to_Teich_curve} there is a biholomorphic map
 \begin{align*}
  \Gamma: S(\Omega,D) & \longrightarrow \pi_T^{-1}(\mathfrak{S}(\Omega))  \\
  (\epsilon,p) & \longmapsto \left( [\riem^P,\nu_\epsilon \circ f, \riem_\epsilon^P], \sigma_\epsilon(p) \right)
 \end{align*}
 where $\sigma_\epsilon: \riem_\epsilon^P \rightarrow \riem_{\hat{\mu}} = \pi_T^{-1}([\riem^P,\nu_\epsilon \circ f, \riem_\epsilon^P])$
 is a biholomorphism depending holomorphically on $\epsilon$, and $f^{-1}_{\hat{\mu}}\circ\sigma_{\epsilon}\circ\nu_{\epsilon(\mu)}\circ f$ is homotopic to the identity, see Remark \ref{rem:relation to teich equivalence}. Thus $[\riem ^P,\nu_{\epsilon(\mu)}\circ f, \riem_\epsilon^P ]=[\riem ^P,\sigma_{\epsilon}\circ\nu_{\epsilon(\mu)}\circ f, \riem_{\hat{\mu}}^P ]=[\riem ^P , f_{\hat{\mu}}, \riem_{\hat{\mu}}^P]$. By restricting $B$ sufficiently (namely, to $\Xi^{-1}(\mathfrak{S}(\Omega))$
 - note that by holomorphicity of $\Xi$ the inverse image of $\mathfrak{S}(\Omega)$ is open), one obtains
 \begin{align} \label{eq:epsilon_of_mu_temp}
   \theta^{-1} \circ \Gamma^{-1}  : \pi_T^{-1}(\Xi(B)) & \longrightarrow \Omega \times \riem^P   \\ \nonumber
  \left([\riem^P,f_{\hat{\mu}},\riem_{\hat{\mu}}^P],w \right)  & \longmapsto    \left( \epsilon(\mu),\nu^{-1}_{\epsilon(\mu)} \circ \sigma^{-1}_{\epsilon(\mu)}(w) \right)
 \end{align}
 where $\Gamma \circ \theta$ is a strong global trivialization. Observe that $\epsilon(\mu)$ is holomorphic
 in $\mu$ (as a function on $H_{-1,1}(\riem)$).

 By Theorem \ref{th:Pi_projection_holomorphic} it suffices to show that
 \begin{align*}
  \Pi \circ \Phi : B & \longrightarrow \widetilde{T}_0(\riem^P)  \\
  \mu & \longmapsto \left[ [\riem^P,f_{\hat{\mu}},\riem_{\hat{\mu}}^P], f_{\hat{\mu}} \circ \tilde{\tau}_1,\ldots,f_{\hat{\mu}} \circ \tilde{\tau}_n \right]
 \end{align*}
 is a biholomorphism. Since $\sigma_{\epsilon(\mu)} ^{-1} \circ f_{\hat{\mu}}$ is homotopic to $\nu_{\epsilon(\mu)} \circ f$ we have by the equivalence relation of Definition \ref{de:rigged teich}
 \[\left[ [\riem^P,f_{\hat{\mu}},\riem_{\hat{\mu}}^P], f_{\hat{\mu}} \circ \tilde{\tau}_1,\ldots,f_{\hat{\mu}} \circ \tilde{\tau}_n \right]=\left[ [\riem^P,\nu_\epsilon \circ f, \riem_\epsilon^P], \sigma^{-1}_{\epsilon}\circ f_{\hat{\mu}} \circ \tilde{\tau}_1,\ldots,\sigma^{-1}_{\epsilon}\circ f_{\hat{\mu}} \circ \tilde{\tau}_n \right]. \]
 
 Therefore, in the coordinates $\mathcal{G}$  defined in (\ref{eq:old_chart_definition}),  $\Pi \circ \Phi$ is written 
 \begin{align*}
  \mathcal{G} \circ \Pi \circ \Phi : W & \rightarrow \mathbb{C}^d \times \Oqcwp \times \cdots \Oqcwp \\
  \mu & \mapsto \left(\epsilon(\mu), \left(
  \zeta_1 \circ \nu_{\epsilon(\mu)}^{-1} \circ \sigma_{\epsilon(\mu)} ^{-1} \circ f_{\hat{\mu}} \circ \tilde{\tau}_1,
  \ldots, \zeta_n \circ \nu_{\epsilon(\mu)}^{-1} \circ \sigma_{\epsilon(\mu)} ^{-1} \circ f_{\hat{\mu}}\circ \tilde{\tau}_n
  \right) \right),
 \end{align*}
where $d$ is the dimension of $T(\riem ^P).$ It was observed above that the first component is holomorphic in $\mu$.  Thus we only need to show that the second
 component is holomorphic in $\mu$, as a map into $\Oqcwp \times \cdots \Oqcwp$.

 To this end, it is enough to show that the second component is G\^ateaux holomorphic and locally bounded (see e.g. \cite{Grosse-Erdmann}).
 Thus we fix $\omega
  \in B$ and consider the maps
 \begin{equation} \label{eq:component}
  t \mapsto \zeta_i \circ \nu_{\epsilon(\omega+t\mu)}^{-1} \circ \sigma_{\epsilon(\omega+t\mu)} ^{-1} \circ f_{\widehat{\omega+t\mu}}
  \circ \tilde{\tau}_i.
 \end{equation}
 where $t$ is restricted to some open neighbourhood of $0 \in \mathbb{C}$ so that $\omega + t \mu \in B$.
   Since  $\Gamma^{-1}$ is holomorphic, and
  $\theta^{-1}$ is holomorphic in $\epsilon$ for fixed $z\in E_i$, we can conclude that $\nu_{\epsilon(\omega+t\mu)}^{-1} \circ \sigma_{\epsilon(\omega+t\mu)} ^{-1}$
  depends holomorphically on $t$.  Furthermore $f_{\hat{\mu}}$ depends holomorphically on $\hat{\mu}$ (because it is a strong local trivialization for the \teich curve as defined in (\ref{eq:f_mu})), and hence
 $f_{\widehat{\omega+t\mu}}$ depends holomorphically on $t$ for fixed $z$.  Since
 \[  \zeta_i \circ \nu_{\epsilon(\omega+t\mu)}^{-1} \circ \sigma_{\epsilon(\omega+t\mu)} ^{-1} \circ f_{\widehat{\omega+t\mu}}
  \circ \tilde{\tau}_i\]
 is also holomorphic in $z$ on $\mathbb{D}$, by Hartogs' theorem it is jointly holomorphic, and thus all $z$ derivatives are holomorphic
 in $t$.
 So
 \[  t \mapsto (\zeta_i \circ \nu_{\epsilon(\omega+t\mu)}^{-1} \circ \sigma_{\epsilon(\omega+t\mu)} ^{-1} \circ f_{\widehat{\omega+t\mu}}
  \circ \tilde{\tau}_i)'(0)  \]
  is holomorphic in $t$.   We now need to show that
  \[  t \mapsto \mathcal{A} \circ \zeta_i \circ \nu_{\epsilon(\omega+t\mu)}^{-1} \circ \sigma_{\epsilon(\omega+t\mu)} ^{-1} \circ f_{\widehat{\omega+t\mu}}  \]
  is holomorphic in $t$ (as a map into $A_2^1(\disk)$) where
 \[  \mathcal{A}(h) = \frac{h''(z)}{h'(z)}.  \]

 Let $e_z:\Oqcwp \rightarrow \mathbb{C}$ denote point evaluation at $z \in E_i$.
  Since the point evaluations are a separating set
 of continuous linear functionals, to show that $\mathcal{G}$ is G\^ateaux holomorphic it is enough (see \cite{Grosse-Erdmann}) to prove that
 \[ e_z \circ \mathcal{A} \circ \zeta_i \circ \nu_{\epsilon(\omega+t\mu)}^{-1} \circ \sigma_{\epsilon(\omega+t\mu)} ^{-1} \circ f_{\widehat{\omega+t\mu}}  \]
 is holomorphic in $t$ for each $z$, and is locally bounded.  The holomorphicity in $t$ for fixed $z$ follows from
 the same argument as above.

Local boundedness will follow from the local boundedness of
 $\mathcal{G} \circ \Pi \circ \Phi$.
 Thus the local boundedness of $\mathcal{G} \circ \Pi \circ \Phi$ is the only
 remaining step.  Recall that by construction, for every $\mu$ in $B$, $\Pi \circ \Phi(\mu) \in F(U,S,\Omega)$;
 in particular, the closure of $\zeta_i \circ \nu_{\epsilon(\mu)}^{-1} \circ \sigma^{-1}_{\epsilon(\mu)}(f_{\hat{\mu}}\circ
 \tilde{\tau}_i(\disk))$
 is in $K_i$.  Thus, since $\theta$ is continuous, if we further restrict $B$ so that $\epsilon(\mu)$ is a
 subset of a compact set $\Omega' \subseteq \Omega$ containing $0$, we can guarantee that $\Gamma^{-1} \circ \pi_T^{-1} \Xi(B)$
 is contained in the compact set
 \[  \{ (\epsilon,\nu_\epsilon(\overline{K_i}))  \,:\, \epsilon \in \Omega' \} \subseteq S(\Omega,D). \]
 This takes care of the local boundedness of $\epsilon(\mu)$.

 Fix an analytic extension of $\tilde{\tau}$ to a disk $D_s = \{ z\,:\,|z|<s \}$ for some $s >1$.
 Since $f_{\hat{\mu}}(z)$ is a continuous function of both $z$ and $\mu$ (again, because it is a strong local trivialization for the \teich curve), and the same thing also holds for $\nu_{\epsilon(\mu)}^{-1} \circ \sigma^{-1}_{\epsilon(\mu)}$, there is a uniform
 $R>1$ such that for every $\mu \in B$ in this ball, we have that
 \[  \nu_{\epsilon(\mu)}^{-1} \circ \sigma^{-1}_{\epsilon(\mu)} \circ f_{\hat{\mu}} \circ \tilde{\tau}_i(D_R) \subseteq E_i  \]
 where $E_i$ is the domain of the $n$-chart $(\zeta_i,E_i)$ and $D_R =\{z\,:\,|z| < R \}$.
 Fix $1<r<R$ and let $F_{\hat{\mu}}$ be any quasiconformal
 extension of $\zeta_i \circ \nu_{\epsilon(\mu)}^{-1} \circ \sigma^{-1}_{\epsilon(\mu)} \circ f_{\hat{\mu}} \circ \tilde{\tau}_i$ to $\mathbb{C}$ agreeing with the original map on $\{z\,:\, |z| \leq r \}$.
 This quasiconformal map represents the same element of $\Oqcwp$.  The $L^2$ norm of the extension satisfies
 \begin{equation}\label{eq: estimate for preschwarzian in terms of L2}
  \iint_{\disk^*} \left| \frac{\overline{\partial} F_{\hat{\mu}}(z)}{\partial F_{\hat{\mu}}(z)}\right|^2
   \frac{1}{(1-|z|^2)^2} \,dA \leq  C \| \mu \|_{2,\riem} +
     \text{ Hyperbolic Area } ( \{z \,:\, |z| >r \} \cup \{\infty \})  
 \end{equation}  
where the first term is a bound on the dilatation in $|z|<r$ arising from Lemma \ref{le:singularity_of_hyp_metric}, and
 the second term bounds the dilatation on $|z|>r$  using only the fact that the complex dilatation of $\tilde{f}_{\hat{\mu}}$ is less
 than one.  It is clear that both terms on the right hand side are finite and bounded
 by a fixed constant.

Referring to (\ref{eq:Oqco_norm_def}) we see that we need to show that $|F_{\hat{\mu}}'(0)|$ and
 \[   \iint_{\disk} \left| \frac{F''_{\hat{\mu}}(z)}{F'_{\hat{\mu}}(z)} \right|^2 \,dA  \]
 are bounded.  Since $\overline{F_{\hat{\mu}}(\disk)}$ is a subset of $K_i$ and in particular
 bounded in some disk, by the Schwarz lemma $|F_{\hat{\mu}}'(0)| \leq M$ for some fixed
 $M >0$.
By Lemma \ref{le:TT_variation}, we have for some $\delta>0$ and $|t| < \delta$,
\[
\iint_{\disk} \left| \frac{F^{''}_{\hat{\mu}}(z)}{F^{'}_{\hat{\mu}}(z)}\right|^2
dA \approx \iint_{\disk} | \mathcal{S} F_{\hat{\mu}}(z)|^2 (1-|z|^2)^2 \, dA + \left| \frac{F''_{\hat{\mu}}(0)}{F'_{\hat{\mu}}(0)} \right|^2
\]
where $S$ denotes the Schwarzian derivative.  By the classical second Taylor coefficient estimate for a univalent map of
$\disk$ the second term on the right hand side is bounded by $4 |F'_{\hat{\mu}}(0)| \leq 4M$.
Finally, by \cite[Theorem 2]{GuoHui} we have the estimate
\[
\iint_{\disk} | \mathcal{S}F_{\hat{\mu}}(z)|^2 (1-|z|^2)^2 \, dA \leq C \iint_{\disk^*}\frac{|\hat{\mu}(z)|^2}{(1-|z|^2)^2} \,dA.
\]
(Note that in \cite[Theorem 2]{GuoHui} the roles of $\disk$ and $\disk^*$ are reversed;
 however the left and right side are invariant under reflection in the circle $z \mapsto 1/z$ so
 it immediately implies the estimate above). Hence by \eqref{eq: estimate for preschwarzian in terms of L2} $\Phi$ is locally
bounded.  This concludes the proof of the theorem.
 \end{proof}

 \begin{theorem}  \label{th:Ahlfors_Weill_inverse} Let $\riem$ be a bordered
  Riemann surface of type $(g,n)$ such that $2g-2+n>0$. There is an open neighbourhood $U$ of $0 \in H_{-1,1}(\riem)$ on which $\Phi$ is a
   biholomorphism.
 \end{theorem}
 \begin{proof}
  We will show first that $D\Phi(0)$ is a topological isomorphism.  Since we have already
  shown that $\Phi$ is holomorphic, it follows that $D\Phi(0)$ is bounded, and
  thus by the open mapping theorem it suffices to show that $D\Phi(0)$ is bijective.

  We first show that $D \Phi(0)$ is surjective.  Let $\mathbf{v}$ be a tangent vector
  to $\twp(\riem)$ at $[\riem,\text{Id},\riem]$.  By Theorem \ref{th:tangent_in_H11}
  we may find a holomorphic curve $\alpha(t)=[\riem,g_t,\riem_t]$ such that the Beltrami differential $\mu_t$
  of $g_t$ is in $H_{-1,1}(\riem)$, has tangent vector $\mathbf{v}$ at $0$ and
  is holomorphic in $t$.   In that case $\alpha(t)=\Phi(\mu_t)$ and so
  \[  \mathbf{v}= \frac{d}{dt} \left. \alpha \right|_{t=0} = D\Phi\left( \left. \frac{d \mu_t}{dt} \right|_{t=0} \right)  \]
  which proves the claim.

  Next we show that $D\Phi(0)$ is injective.  The kernel is trivial since
  \begin{align*}
   \mathrm{ker}\, D\Phi(0) & = \mathrm{ker} \,D \check{\Phi}(0) \cap H_{-1,1}(\riem) \\
   & = \left( \mathrm{ker} \, D \check{\Phi}(0) \cap \Omega_{-1,1}(\riem) \right) \cap H_{-1,1}(\riem) \\
   & = \{0 \}
  \end{align*}
  by Theorem \ref{th:decomp_on_riemB} and Corollary \ref{co:intersection_good}.

  Since $\Phi$
  is holomorphic by Theorem \ref{thm:holo_H11_coords}, by the inverse function
  theorem $\Phi$ has a $C^1$ inverse in a neighbouhood of $0$.  Thus there is a $U$ on
  which $\Phi$ is a biholomorphism.
 \end{proof}
 Finally, we show that $\twp(\riem)$ possesses an atlas of charts modelled on $H_{-1,1}(\riem)$.
 We require a change of base point.  Let $[\riem,f,\riem_0] \in \twp(\riem)$.  Define the
 change of base point map by
 \begin{align*}
  f^*: \twp(\riem_0) & \longrightarrow \twp(\riem) \\
  [\riem_0,\rho,\riem_1] & \longmapsto [\riem,\rho \circ f,\riem_1].
 \end{align*}
 It was shown in \cite[Section 5.3]{RSS_Hilbert} that this map is a biholomorphism.  Thus we may
 conclude that
 \begin{theorem}\label{th:last bloody theorem} Let $\riem$ be a bordered Riemann surface of type $(g,n)$
 $2g-2+n>0$.  For each point
 $[\riem,f,\riem_1]$, there is an open neighbourhood $B$ of $0$ in $H_{-1,1}(\riem)$ such
 that
 \begin{align*}
   \Phi_{(\riem,f,\riem_0)}: B & \longrightarrow \twp(\riem) \\
   \mu & \longmapsto [\riem, f_\mu \circ f,\riem_1]
 \end{align*}
 is a biholomorphism onto an open neighbourhood of $[\riem,f,\riem_0]$.  In particular,
 the collection of such biholomorphisms forms an atlas on $\twp(\riem)$.
 \end{theorem}
 \begin{proof}
  This follows immediately from Theorem \ref{th:Ahlfors_Weill_inverse} and the fact that $f^*$
  is a biholomorphism.
 \end{proof}
 Note that the map $f^*$ is independent of the choice of representative of $[\riem,f,\riem_1]$.
 \end{subsection}
\begin{subsection}{The explicit Weil-Petersson metric}\label{explicit WP}
 We are now ready to define the Weil-Petersson metric on the tangent space at the identity, which is done as follows.
 Any pair of tangent vectors can be represented by a pair of Beltrami differentials
 $\mu,\nu \in H_{-1,1}(\riem)$.
 Let $\mu, \nu \in H_{-1,1}(\riem)$ be two representatives of elements of the tangent space
 at the identity of the refined Teichm\"uller space of $\riem^B$.  For coordinates on an open
 set $U$ containing the
 set $W$
 we define the local integral as in Section \ref{se:differentials}.   Assuming
 that $\mu = \mu_U d\bar{z}/dz$ and $\nu=\nu_U d\bar{z}/dz$ in local coordinates,
 if $W$ is a measurable set contained in $U$ we can define the
 integral
 \[   \iint_{W} \mu_U(z) \overline{\nu}_U(z) \rho_U(z)^2 \,dA  \]
 where $\rho_U(z)^2|dz|^2$ is the local expression for the hyperbolic metric on $\riem$ in the
 chart $U$.
 It is easily checked that this expression is invariant under change of coordinates.
 Using a partition of unity subordinate to a cover we can extend this to an integral
 over $\riem$, which we will denote by
 \begin{equation} \label{eq:WP_metric_Beltrami}
   \left<\mu,\nu \right>_{[\riem,\text{Id},\riem]} = \iint_{\riem} \mu  \overline{\nu}  \rho_\riem^2
 \end{equation}
  where $\rho_\riem^2$ is the hyperbolic area form on $\riem$.

 One may also represent tangent space elements as lying in the space of quadratic differentials
 $A^2_2(\riem)$; that is for quadratic differentials $\alpha, \beta \in A^2_2(\riem)$
 given by
 \[  \alpha = \mathfrak{B}^{-1}(\mu) \ \ \, \beta=\mathfrak{B}^{-1}(\nu)  \]
 we can define the integral
 \[  \left( \alpha, \beta \right)_{[\riem,\text{Id},\riem]} = \left< \mathfrak{B}(\alpha),\mathfrak{B}(\beta) \right>_{[\riem,\text{Id},\riem]}.  \]
 %{\color{red}This integral can also be defined in the same (standard) way as in (\ref{eq:WP_metric_Beltrami}),
%  and can be expressed as
% \[  \left< \alpha, \beta \right> =  \iint_\riem \alpha \overline{\beta} \rho_\riem^{-2}\, dA. \]}

Finally, we observe that by changing the base point using $f^*$, we may define the Weil-Petersson
 metric at arbitrary points as follows.  For a change
 of base point map $f^*$ denote its derivative by $Df^*$.
 \begin{definition}  Let $\riem$ be a bordered Riemann surface of type $(g,n)$
 and let $[\riem,f,\riem_1] \in \twp(\riem)$.  For $\mathbf{v},\mathbf{w}  \in T_{[\riem,f,\riem_0]} \twp(\riem)$
 define the generalized Weil-Petersson metric by
 \[  \left< \mathbf{v},\mathbf{w} \right>_{[\riem,f,\riem_0]} =
    \left< D(f^{-1})^* \mathbf{v}, D(f^{-1})^* \mathbf{w} \right>_{[\riem_0,\text{Id},\riem_0]}.   \]
 \end{definition}
 Since $f^*$ is independent of the representative $[\riem,f,\riem_1]$, this is well-defined. 
 
 One may define a similar expression in terms of the quadratic differentials.

 \begin{remark}
  There is no transitive group action on $\twp(\riem)$ to make use of, and
  we cannot lift the picture to the cover.
  Thus we cannot hope to say what ``right invariance'' even means.
  This is the unique metric on $\twp(\riem)$ which is invariant under change of
  base point and agrees with (\ref{eq:WP_metric_Beltrami}) for a single choice
  of base point.
 \end{remark}
\end{subsection}
\end{section}

\end{document}